\documentclass[11pt,a4paper]{article}

\usepackage[T1]{fontenc}
\usepackage[utf8]{inputenc}
\usepackage[english]{babel}
\usepackage{lmodern}
\usepackage{microtype}
\usepackage[margin=28mm]{geometry}

\usepackage{amsmath,amssymb,amsthm,mathtools,mathrsfs}
\usepackage{enumitem}
\usepackage[numbers,sort&compress]{natbib}
\usepackage{xcolor}

\usepackage[
  colorlinks=true,
  linkcolor=blue!60!black,
  citecolor=blue!60!black,
  urlcolor=blue!60!black
]{hyperref}

\hypersetup{
  pdftitle={
    Local Approximation Defects for Caputo Differentiation
  },
  pdfsubject={
    Linear dynamics and fractional differentiation
  }
}

\numberwithin{equation}{section}

\newtheorem{theorem}{Theorem}[section]
\newtheorem{proposition}[theorem]{Proposition}
\newtheorem{lemma}[theorem]{Lemma}
\newtheorem{corollary}[theorem]{Corollary}

\theoremstyle{definition}
\newtheorem{definition}[theorem]{Definition}

\theoremstyle{remark}

\newtheorem{example}[theorem]{Example}

\newcommand{\Nzero}{\mathbb N_0}
\newcommand{\C}{\mathbb C}

\newcommand{\X}{X_{\alpha,\rho}^{p}}
\newcommand{\F}{\mathcal F_\alpha}
\newcommand{\CD}{\mathcal C_\alpha}
\newcommand{\Pol}{\mathscr P_\alpha}

\newcommand{\norm}[1]{\left\lVert#1\right\rVert}

\DeclareMathOperator{\spanop}{span}
\DeclareMathOperator{\Per}{Per}
\DeclareMathOperator{\dist}{dist}
\DeclareMathOperator{\Orb}{Orb}

\setlist[enumerate]{
  label=\textup{(\roman*)},
  leftmargin=*
}

\title{Caputo Dynamics: Optimal Approximation and Chaos}

\author{El-Mahdi Nafia, Aziz El Ghazouani, M'hamed El Omari}

\begin{document}

\maketitle

\begin{abstract}
We study bounded realizations of Caputo differentiation
on coefficient Banach spaces of fractional power series.
The main contribution is an explicit quantitative analysis
of local approximation and periodic reconstruction.
We determine the least possible simultaneous error
in the initial displacement and the prescribed target.
The optimal value is attained, and a minimizing
perturbation is given explicitly.
We also solve the exact-arrival problem and establish
uniqueness of its least-norm solution.
These formulas provide computable approximation times
and identify the sharp threshold for hypercyclicity,
mixing and Devaney chaos.
Every periodic point is described by a finite block
whose geometric repetition yields exact norm formulas.
The resulting approximations attain any prescribed
sufficiently large minimal period near any given vector.
For distinct positive powers, a simultaneous block
construction gives quantitative error estimates
and the same sharp threshold for disjoint mixing.
Compact embeddings relate the Banach realizations
to the Mittag--Leffler--Caputo Fr\'echet space.
They also exhibit dense stable invariant subspaces
inside chaotic realizations.
The local approximation condition is a norm formulation
of classical transitivity. Its quantitative evaluation,
together with explicit periodic and simultaneous
reconstruction, provides the central results of the study.
\end{abstract}

\medskip


% ============================================================
\section{Introduction}
\label{sec:introduction}
% ============================================================

Caputo differentiation becomes a dynamical operator
only after its domain and topology have been specified.
The initial endpoint also requires attention.
For instance, fractional powers are absolutely
continuous on compact intervals, although their
ordinary derivatives may be singular at zero.
The standard analytic framework is developed in
\cite{Podlubny,Samko,Gorenflo}.
Coefficient representations provide useful spaces
for studying these operators.
The reproducing-kernel constructions of Rosenfeld
and Dixon \cite{RD} and Rosenfeld, Russo and Dixon
\cite{RRD} are relevant examples.

Backward shifts supply a basic model for dense-orbit
phenomena on Banach spaces.
Rolewicz \cite{Rolewicz} established hypercyclicity
for suitable scalar multiples of the unilateral
backward shift.
Later results characterized the dynamics of weighted
shifts through their weights and underlying sequence
spaces \cite{Salas,GEWeightedShifts}.
Kitai's criterion \cite{Kitai} and the Hypercyclicity
Criterion of Gethner and Shapiro \cite{GS} give
widely used sufficient conditions;
see \cite{BM,GEP} for systematic treatments.
B\`es and Peris \cite{BP} proved that the latter
criterion is equivalent to hypercyclicity of the
direct sum of two copies of the operator.
The construction of de la Rosa and Read
\cite{DeLaRosaRead} shows that this property is
stronger than hypercyclicity alone.

Mixing and Kitai's criterion must likewise be
distinguished.
Grivaux \cite{Grivaux} constructed mixing operators
that fail this criterion.
Gomes and Grosse-Erdmann \cite{GG} characterized
Kitai's criterion for invertible composition operators
in a measurable-function setting and obtained
a mixing composition operator outside its scope.
For the Caputo realizations considered here,
a shift conjugacy already determines the qualitative
dynamics.
The remaining question is quantitative:
how small can an initial perturbation be when
a target and an iteration time are prescribed?

The closest fractional antecedents concern continuous
realizations and discrete approximation schemes.
Lizama, Murillo-Arcila and Peris \cite{LMP}
studied chaotic nonlocal operators associated with
time-stepping methods.
Their approach relates the discrete operators
to Toeplitz operators on sequence spaces.
Murillo-Arcila, Peris and Vargas-Moreno \cite{MPV}
established Devaney chaos for complex Caputo
differentiation on a Mittag--Leffler--Caputo
Fr\'echet space.
They also proved disjoint mixing for powers of
nonzero scalar multiples and investigated
the L1 discretization.
Their continuous realization provides the principal
point of comparison for this paper.

We study fractional power series equipped with
a coefficient Banach norm.
For
\[
0<\alpha<1,\qquad
\rho>0,\qquad
1\le p<\infty,
\]
the space $X_{\alpha,\rho}^{p}$ is defined through
the normalized functions
\[
e_n(t)
=
\frac{\rho^n t^{\alpha n}}
{\Gamma(\alpha n+1)},
\qquad n\ge0.
\]
The coefficient map
\[
J_{\alpha,\rho}:
\ell^p(\mathbb N_0)\longrightarrow X_{\alpha,\rho}^{p}
\]
is an isometric isomorphism.
We justify the Caputo integral for every element
of this space and define its endpoint value
by continuous extension.
The resulting bounded operator satisfies
\[
J_{\alpha,\rho}^{-1}\mathcal C_\alpha J_{\alpha,\rho}
=
\rho B,
\]
where $B$ is the unilateral backward shift.
This representation fixes the norm in which
approximation costs can be calculated.

Our main results solve two optimization problems
at a prescribed time.
For a bounded operator $A$ on a Banach space $X$,
define
\[
\Delta_{A,n}(f,g)
=
\inf_{h\in X}
\max\left\{
\norm{h},
\norm{A^n(f+h)-g}
\right\},
\qquad n\ge0.
\]
The two terms measure the initial displacement
and the terminal error.
On a separable infinite-dimensional Banach space,
the condition
\[
\inf_{n\ge0}\Delta_{A,n}(f,g)=0
\qquad(f,g\in X)
\]
is equivalent to hypercyclicity.
This is Birkhoff transitivity expressed through
norm balls \cite{BM,GEP}.
For our realizations, the fixed-time defect
has an explicit attained value.

Let $A=\lambda\mathcal C_\alpha$, where $\lambda\ne0$.
We prove
\begin{equation}
\label{eq:intro-local-cost}
\Delta_{A,n}(f,g)
=
\frac{\norm{g-A^nf}_{X_{\alpha,\rho}^{p}}}
{1+(|\lambda|\rho)^n}.
\end{equation}
An explicit perturbation attains this value
with equal initial and terminal errors.
More generally, for $\varepsilon,\eta\ge0$,
the constraints
\[
\norm{h}_{X_{\alpha,\rho}^{p}}\le\varepsilon,
\qquad
\norm{A^n(f+h)-g}_{X_{\alpha,\rho}^{p}}\le\eta
\]
are feasible precisely when
\[
\norm{g-A^nf}_{X_{\alpha,\rho}^{p}}
\le
(|\lambda|\rho)^n\varepsilon+\eta.
\]
When the terminal error must vanish, we obtain
\begin{equation}
\label{eq:intro-exact-cost}
\min_{A^n(f+h)=g}
\norm{h}_{X_{\alpha,\rho}^{p}}
=
(|\lambda|\rho)^{-n}
\norm{g-A^nf}_{X_{\alpha,\rho}^{p}}.
\end{equation}
The minimizing exact perturbation is unique
and is given explicitly.

These formulas complement the Caputo dynamics
established in \cite{MPV}.
The Fr\'echet topology used there describes
approximation through neighbourhoods determined
by finitely many graph norms.
In the present setting,
\eqref{eq:intro-local-cost}--\eqref{eq:intro-exact-cost}
determine optimal costs in one specified Banach norm.
Hypercyclicity, mixing and Devaney chaos occur
exactly when
\[
|\lambda|\rho>1.
\]
The threshold follows from the shift conjugacy
and records the dependence on the chosen realization.
Compact dense embeddings connect these Banach spaces
to the Mittag--Leffler--Caputo Fr\'echet space.
They also recover its corresponding dynamical
properties for nonzero scalar multiples
within the range $0<\alpha<1$ studied here.

Above this threshold, periodic points can be
reconstructed from finite coefficient blocks.
Their geometric repetition gives exact norm formulas
and bounded finite-rank projections onto
the periodic subspaces.
These projections converge strongly to the identity
with explicit approximation estimates.
We also show that every nonempty open set contains
points of every sufficiently large minimal period.
Mittag--Leffler expansions express the periodic blocks
as finite combinations of eigenvectors,
linking this reconstruction to the eigenvector
methods of \cite{GodefroyShapiro,MPV}.

For distinct positive powers, a separated-block
construction controls several targets simultaneously.
The resulting estimates involve coefficient tails,
geometric decay and the gaps between the powers.
They yield the same sharp threshold for disjoint
hypercyclicity and disjoint mixing.
Relative to the theory developed in
\cite{BPeris,BMPS}, the construction makes
the simultaneous approximation errors explicit.

Recent work provides further structural context.
Bernardes and Peris \cite{BernardesPeris2024}
relate shadowing, chain recurrence and mixing,
with equivalences under finite shadowing assumptions.
Adhikary and Nagar \cite{AN2026} formulate
transitivity criteria using Furstenberg families
and obtain characterizations of mixing.
Cobollo, Ernst, Menet and Peris \cite{CEMP2026}
introduce a targeting property for linearizations
and apply it to certain Lipschitz-free operators.
These approaches address conditions ensuring
dynamical properties in their respective settings.
Here the prescribed-time problem admits
optimal values and explicit minimizing perturbations.

For $p=2$, the least-norm exact perturbation
is given by the classical Moore--Penrose inverse
\cite{MalikThome}.
Our coefficient proof applies to every
$1\le p<\infty$, including $p=1$,
and establishes uniqueness through disjoint supports.
An elementary right-inverse argument determines
the balanced approximation cost.
The contribution combines these optimization formulas
with periodic reconstruction, simultaneous approximation
and quantitative embeddings for the bounded
Caputo realizations.

The paper is organized as follows.
Section~\ref{sec:preliminaries} collects the dynamical
and fractional-calculus notions.
Section~\ref{sec:caputo-realization} constructs
the coefficient spaces and the bounded
Caputo operator.
Section~\ref{sec:local-defects} solves the balanced
and exact-arrival problems and describes
the spectral data.
Section~\ref{sec:periodic-chaos} treats periodic
reconstruction, prescribed large minimal periods
and the sharp dynamical classification.
Section~\ref{sec:disjoint-powers} develops simultaneous
approximation for distinct powers.
Section~\ref{sec:banach-scales} determines optimal
finite-rank approximation errors for the Banach
inclusions and compares these realizations
with the Fr\'echet space.
Section~\ref{sec:examples} gives explicit costs,
approximation times and fractional cycles.
Section~\ref{sec:conclusion} states the conclusions
and possible extensions.
% ============================================================
\section{Preliminaries}
\label{sec:preliminaries}

Throughout, $X$ denotes a separable infinite-dimensional
complex Banach space, $\mathcal L(X)$ its bounded linear
operators, and $\Nzero=\{0,1,2,\ldots\}$.

% ============================================================
\subsection{Linear dynamics}
% ============================================================

\begin{definition}\cite{BM,GEP}
\label{def:hypercyclic}
For $A\in\mathcal L(X)$ and $x\in X$, the orbit of $x$
under $A$ is
\[
\Orb(x,A)=\{A^nx:n\in\Nzero\}.
\]
The vector $x$ is hypercyclic for $A$ if
\[
\overline{\Orb(x,A)}=X.
\]
The operator $A$ is hypercyclic if it admits such a vector.
We denote its set of hypercyclic vectors by $HC(A)$.
\end{definition}

\begin{definition}\cite{BM,GEP}
\label{def:mixing}
An operator $A\in\mathcal L(X)$ is topologically transitive
if, for every pair of nonempty open sets $U,V\subset X$,
there exists $n\in\Nzero$ such that
\[
A^n(U)\cap V\ne\varnothing.
\]
It is topologically mixing if, for every such pair,
there exists $n_0\in\Nzero$ such that
\[
A^n(U)\cap V\ne\varnothing
\qquad(n\ge n_0).
\]
It is weakly mixing if $A\oplus A$ is topologically
transitive on $X\times X$.
\end{definition}

\begin{theorem}\cite{BM,GEP}
\label{thm:birkhoff}
For $A\in\mathcal L(X)$, hypercyclicity and topological
transitivity are equivalent. When these properties hold,
$HC(A)$ is a dense $G_\delta$ subset of $X$.
\end{theorem}

\begin{definition}\cite{Banks,GEP}
\label{def:chaos}
A vector $x\in X$ is periodic for $A\in\mathcal L(X)$
if $A^mx=x$ for some integer $m\ge1$. Its minimal period
is the least such integer. The set of periodic points is
\[
\Per(A)=\bigcup_{m\ge1}\ker(A^m-I).
\]

The operator $A$ is sensitive if there exists $\delta>0$
such that, for every $x\in X$ and $\varepsilon>0$,
one can find $y\in X$ and $n\in\Nzero$ satisfying
\[
\norm{x-y}<\varepsilon,
\qquad
\norm{A^nx-A^ny}>\delta.
\]

The operator $A$ is Devaney chaotic if it is topologically
transitive, sensitive, and has dense periodic points.
\end{definition}

\begin{proposition}\cite{Banks,GEP}
\label{prop:chaos-characterization}
An operator $A\in\mathcal L(X)$ is Devaney chaotic
if and only if it is hypercyclic and
\[
\overline{\Per(A)}=X.
\]
\end{proposition}

\begin{theorem}\cite{Kitai,GEP}
\label{thm:kitai}
Let $A\in\mathcal L(X)$. Suppose that there exist dense
sets $D_1,D_2\subset X$ and a map $S:D_2\to D_2$ such that
\[
A^nx\longrightarrow0
\qquad(x\in D_1),
\]
\[
S^ny\longrightarrow0
\qquad(y\in D_2),
\]
and
\[
ASy=y
\qquad(y\in D_2).
\]
Then $A$ is topologically mixing and therefore hypercyclic.
\end{theorem}

\begin{theorem}\cite{GS,BP,GEP}
\label{thm:hc-criterion}
Let $A\in\mathcal L(X)$. Suppose that there exist dense
sets $D_1,D_2\subset X$, a strictly increasing sequence
$(n_k)$ of positive integers, and maps $R_k:D_2\to X$
such that
\[
A^{n_k}x\longrightarrow0
\qquad(x\in D_1),
\]
and
\[
R_ky\longrightarrow0,
\qquad
A^{n_k}R_ky\longrightarrow y
\qquad(y\in D_2).
\]
Then $A$ is weakly mixing and hypercyclic.
If these conditions hold with $n_k=k$, then $A$ is
topologically mixing.
\end{theorem}

\begin{definition}\cite{BPeris,BMPS}
\label{def:disjoint}
Let $N\ge2$ and $A_1,\ldots,A_N\in\mathcal L(X)$.
The tuple $(A_1,\ldots,A_N)$ is disjoint hypercyclic
if there exists $x\in X$ such that
\[
\overline{
\left\{
(A_1^nx,\ldots,A_N^nx):n\in\Nzero
\right\}}
=X^N.
\]

The tuple is disjoint mixing if, for all nonempty open
sets $U_0,\ldots,U_N\subset X$, there exists
$n_0\in\Nzero$ such that
\[
U_0\cap\bigcap_{j=1}^{N}A_j^{-n}(U_j)
\ne\varnothing
\qquad(n\ge n_0).
\]
\end{definition}

% ============================================================
\subsection{Fractional calculus and Mittag--Leffler functions}
% ============================================================

\begin{definition}\cite{Podlubny,Samko}
\label{def:fractional-integral}
For $\gamma>0$, the left Riemann--Liouville fractional
integral is
\[
(I_{0+}^{\gamma}f)(t)
=
\frac{1}{\Gamma(\gamma)}
\int_0^t(t-s)^{\gamma-1}f(s)\,ds,
\]
whenever the integral is absolutely convergent.
\end{definition}

\begin{definition}\cite{Podlubny,Samko}
\label{def:caputo}
Let $0<\alpha<1$ and $f\in AC[0,R]$.
The Caputo derivative of order $\alpha$ is
\begin{equation}
\label{eq:caputo-definition}
({}^CD_{0+}^{\alpha}f)(t)
=
(I_{0+}^{1-\alpha}f')(t)
=
\frac{1}{\Gamma(1-\alpha)}
\int_0^t(t-s)^{-\alpha}f'(s)\,ds,
\end{equation}
at points $t\in(0,R]$ where the integral is absolutely
convergent.
\end{definition}

\begin{lemma}\cite{Podlubny,Samko}
\label{lem:powers}
For $\gamma>0$, $\beta\ge0$ and $t>0$,
\[
I_{0+}^{\gamma}t^\beta
=
\frac{\Gamma(\beta+1)}
{\Gamma(\beta+\gamma+1)}
t^{\beta+\gamma}.
\]
For $0<\alpha<1$ and $\beta>0$,
\[
{}^CD_{0+}^{\alpha}t^\beta
=
\frac{\Gamma(\beta+1)}
{\Gamma(\beta+1-\alpha)}
t^{\beta-\alpha},
\qquad
{}^CD_{0+}^{\alpha}1=0.
\]
\end{lemma}

\begin{proposition}\cite{Podlubny,Samko}
\label{prop:fractional-inversion}
Let $0<\alpha<1$ and $f\in AC[0,R]$.
Then
\[
I_{0+}^{\alpha}
\bigl({}^CD_{0+}^{\alpha}f\bigr)(t)
=
f(t)-f(0)
\]
for almost every $t\in(0,R)$.

If $g\in C^1[0,R]$, then
$I_{0+}^{\alpha}g\in AC[0,R]$ and
\[
{}^CD_{0+}^{\alpha}
\bigl(I_{0+}^{\alpha}g\bigr)(t)
=
g(t),
\qquad 0<t\le R.
\]
\end{proposition}

\begin{definition}\cite{Gorenflo}
\label{def:mittag}
For $\eta,\beta>0$, the two-parameter Mittag--Leffler
function is
\[
E_{\eta,\beta}(z)
=
\sum_{n=0}^{\infty}
\frac{z^n}{\Gamma(\eta n+\beta)},
\qquad z\in\C.
\]
We write
\[
E_\eta(z)=E_{\eta,1}(z).
\]
Both functions are entire.
\end{definition}% % ============================================================
\section{Coefficient spaces and the Caputo realization}
\label{sec:caputo-realization}
% ============================================================

Throughout this section,
\[
0<\alpha<1,
\qquad
\rho>0,
\qquad
1\le p<\infty.
\]
We first construct the coefficient space. We then identify
Caputo differentiation and fractional integration as bounded
operators on that space.

% ------------------------------------------------------------
\subsection{The coefficient space}
% ------------------------------------------------------------

\begin{definition}[Fractional coefficient space]
\label{def:space}
For $n\in\Nzero$, set
\begin{equation}
\label{eq:basis}
e_n(t)
=
\frac{\rho^n t^{\alpha n}}{\Gamma(\alpha n+1)},
\qquad t\ge0.
\end{equation}
Define
\begin{equation}
\label{eq:space}
\X
=
\left\{
f:[0,\infty)\to\C:
f(t)=\sum_{n=0}^{\infty}a_ne_n(t),
\quad
a=(a_n)_{n\ge0}\in\ell^p(\Nzero)
\right\}.
\end{equation}
The coefficient norm is
\begin{equation}
\label{eq:coefficient-norm}
\norm{f}_{\X}
=
\left(\sum_{n=0}^{\infty}|a_n|^p\right)^{1/p}.
\end{equation}
\end{definition}

The next result makes this definition well posed.

\begin{lemma}[Convergence, uniqueness and completeness]
\label{lem:space}
Every series in \eqref{eq:space} converges absolutely
and uniformly on compact intervals and has unique
coefficients. The map
\[
J_{\alpha,\rho}:\ell^p(\Nzero)\longrightarrow\X,
\qquad
J_{\alpha,\rho}a=\sum_{n=0}^{\infty}a_ne_n,
\]
is an isometric isomorphism.

Consequently, $\X$ is a separable infinite-dimensional
Banach space, $(e_n)_{n\ge0}$ is a normalized unconditional
Schauder basis, and
\[
\Pol
=
\spanop\{e_n:n\ge0\}
=
\spanop\{1,t^\alpha,t^{2\alpha},\ldots\}
\]
is dense in $\X$.
\end{lemma}

\begin{proof}
For $R>0$, the estimate
$\norm{a}_\infty\le\norm{a}_p$ gives
\begin{equation}
\label{eq:evaluation-bound}
\sup_{0\le t\le R}
\sum_{n=0}^{\infty}|a_n|e_n(t)
\le
\norm{a}_p
\sum_{n=0}^{\infty}
\frac{\rho^nR^{\alpha n}}{\Gamma(\alpha n+1)}
=
\norm{a}_p E_\alpha(\rho R^\alpha).
\end{equation}
The Weierstrass test proves the asserted convergence.

To establish uniqueness, consider
\[
H_a(w)
=
\sum_{n=0}^{\infty}
\frac{a_n\rho^n}{\Gamma(\alpha n+1)}w^n.
\]
Stirling's formula yields
\[
\Gamma(\alpha n+1)^{1/n}
\sim
\left(\frac{\alpha n}{e}\right)^\alpha,
\]
and therefore
\[
\limsup_{n\to\infty}
\left|
\frac{a_n\rho^n}{\Gamma(\alpha n+1)}
\right|^{1/n}
=0.
\]
Thus $H_a$ is entire.

If two coefficient sequences $a,b\in\ell^p$ define
the same function, then
\[
H_a(t^\alpha)=H_b(t^\alpha)
\qquad(t\ge0).
\]
The identity theorem gives $H_a=H_b$, hence $a_n=b_n$
for every $n$.

The norm in \eqref{eq:coefficient-norm} is therefore
well defined, and $J_{\alpha,\rho}$ is a linear
isometric bijection. Completeness, separability and
the basis assertions follow from those of $\ell^p$.
\end{proof}

We record the local regularity needed for the Caputo integral.

\begin{lemma}[Absolute continuity]
\label{lem:absolute-continuity}
Let $f=J_{\alpha,\rho}a\in\X$.
For every $R>0$, one has $f\in AC[0,R]$ and
\begin{equation}
\label{eq:ordinary-derivative}
f'(t)
=
\sum_{n=1}^{\infty}
a_n
\frac{\rho^n t^{\alpha n-1}}{\Gamma(\alpha n)}
\quad\text{in }L^1(0,R).
\end{equation}
Moreover,
\begin{equation}
\label{eq:AC-bound}
\norm{f'}_{L^1(0,R)}
\le
\sum_{n=1}^{\infty}|a_n|e_n(R)
\le
\norm{f}_{\X}
\bigl(E_\alpha(\rho R^\alpha)-1\bigr).
\end{equation}
\end{lemma}

\begin{proof}
For $n\ge1$,
\[
e_n'(t)
=
\frac{\rho^n t^{\alpha n-1}}{\Gamma(\alpha n)}
\ge0,
\qquad
\int_0^R e_n'(t)\,dt=e_n(R).
\]
Hence
\[
\sum_{n=1}^{\infty}
\norm{a_ne_n'}_{L^1(0,R)}
=
\sum_{n=1}^{\infty}|a_n|e_n(R)
<\infty.
\]
The derivative series converges in $L^1(0,R)$
to a function $v$. Since integration is continuous
from $L^1(0,R)$ into $C[0,R]$,
\[
a_0+\int_0^t v(s)\,ds
=
a_0+\sum_{n=1}^{\infty}a_ne_n(t)
=
f(t).
\]
Thus $f\in AC[0,R]$ and $f'=v$ almost everywhere.
The preceding estimates give \eqref{eq:AC-bound}.
\end{proof}

Point evaluations also provide quantitative control
of compact restrictions.

\begin{proposition}[Evaluations and compact restrictions]
\label{prop:evaluation}
Let $q$ satisfy $1/p+1/q=1$, with $q=\infty$ when $p=1$.
For every $t\ge0$, the functional $\delta_t(f)=f(t)$
is continuous and
\begin{equation}
\label{eq:evaluation-exact}
\norm{\delta_t}_{(\X)^*}
=
\norm{(e_n(t))_{n\ge0}}_{\ell^q}.
\end{equation}

For $R>0$, the restriction operator
\[
\mathcal R_R:\X\longrightarrow C[0,R],
\qquad
\mathcal R_Rf=f|_{[0,R]},
\]
is compact and satisfies
\begin{equation}
\label{eq:restriction-norm}
\norm{\mathcal R_R}
=
\norm{(e_n(R))_{n\ge0}}_{\ell^q}
\le
E_\alpha(\rho R^\alpha).
\end{equation}

If
\[
P_Nf=\sum_{n=0}^{N-1}a_ne_n,
\qquad N\ge1,
\]
then
\begin{equation}
\label{eq:restriction-tail}
\norm{\mathcal R_R-\mathcal R_RP_N}
=
\norm{(e_n(R))_{n\ge N}}_{\ell^q}
\longrightarrow0.
\end{equation}
\end{proposition}

\begin{proof}
Under $J_{\alpha,\rho}$, evaluation at $t$ becomes
\[
a\longmapsto\sum_{n=0}^{\infty}a_ne_n(t).
\]
Duality between $\ell^p$ and $\ell^q$ gives
\eqref{eq:evaluation-exact}.

Since $e_n(t)\le e_n(R)$ for $0\le t\le R$,
H\"older's inequality yields
\[
\norm{\mathcal R_Rf}_\infty
\le
\norm{a}_p
\norm{(e_n(R))_{n\ge0}}_{\ell^q}.
\]
Evaluation at $R$ gives the reverse operator-norm
inequality, proving \eqref{eq:restriction-norm}.

Applying the same argument to coefficient tails gives
\eqref{eq:restriction-tail}. The sequence
$(e_n(R))_{n\ge0}$ belongs to $\ell^1$ and tends to zero,
so its $\ell^q$ tails vanish, including when $q=\infty$.
Thus $\mathcal R_R$ is an operator-norm limit of
finite-rank operators.
\end{proof}

% ------------------------------------------------------------
\subsection{Fractional differentiation and integration}
% ------------------------------------------------------------

We now identify the integral action with a coefficient shift.

\begin{proposition}[Bounded Caputo realization]
\label{prop:caputo}
For every $f=J_{\alpha,\rho}a\in\X$ and $t>0$,
the integral
\begin{equation}
\label{eq:caputo-realization}
(\CD f)(t)
=
\frac{1}{\Gamma(1-\alpha)}
\int_0^t(t-s)^{-\alpha}f'(s)\,ds
\end{equation}
is absolutely convergent.

Its continuous extension to $t=0$ belongs to $\X$,
and
\begin{equation}
\label{eq:shift}
\CD f
=
\rho\sum_{j=0}^{\infty}a_{j+1}e_j.
\end{equation}
In particular,
\begin{equation}
\label{eq:caputo-conjugacy}
J_{\alpha,\rho}^{-1}\CD J_{\alpha,\rho}
=
\rho B,
\qquad
B(a_0,a_1,\ldots)=(a_1,a_2,\ldots).
\end{equation}

For every $m\in\Nzero$,
\begin{equation}
\label{eq:caputo-iterates}
\CD^mf
=
\rho^m\sum_{j=0}^{\infty}a_{m+j}e_j,
\end{equation}
and
\begin{equation}
\label{eq:normpowers}
\norm{\CD^mf}_{\X}^p
=
\rho^{mp}\sum_{j=m}^{\infty}|a_j|^p,
\qquad
\norm{\CD^m}=\rho^m.
\end{equation}
\end{proposition}

\begin{proof}
The Caputo derivative of $e_0=1$ is zero.
For $n\ge1$, the beta integral gives
\[
\begin{aligned}
&\frac{1}{\Gamma(1-\alpha)}
\int_0^t(t-s)^{-\alpha}e_n'(s)\,ds\\
&\quad=
\frac{\rho^n}{\Gamma(1-\alpha)\Gamma(\alpha n)}
\int_0^t
(t-s)^{-\alpha}s^{\alpha n-1}\,ds\\
&\quad=
\frac{\rho^n t^{\alpha(n-1)}}
{\Gamma(\alpha(n-1)+1)}
=
\rho e_{n-1}(t).
\end{aligned}
\]

By Lemma~\ref{lem:absolute-continuity},
$f'=\sum_{n\ge1}a_ne_n'$ almost everywhere.
For fixed $t>0$,
\[
\begin{aligned}
&\frac{1}{\Gamma(1-\alpha)}
\sum_{n=1}^{\infty}|a_n|
\int_0^t(t-s)^{-\alpha}e_n'(s)\,ds\\
&\qquad=
\rho\sum_{n=1}^{\infty}|a_n|e_{n-1}(t)\\
&\qquad\le
\rho\norm{a}_\infty E_\alpha(\rho t^\alpha)
<\infty.
\end{aligned}
\]
Tonelli's theorem therefore justifies integration
of the derivative series term by term.
This proves absolute convergence of
\eqref{eq:caputo-realization} and formula
\eqref{eq:shift}.

The shifted coefficients belong to $\ell^p$.
Lemma~\ref{lem:space} then gives compact uniform
convergence of the resulting series and its
continuous extension to zero.

Iterating \eqref{eq:shift} proves
\eqref{eq:caputo-iterates} and the first identity
in \eqref{eq:normpowers}. Hence
$\norm{\CD^m}\le\rho^m$.
Since
\[
\norm{e_m}_{\X}=1,
\qquad
\CD^me_m=\rho^me_0,
\]
the operator norm equals $\rho^m$.
\end{proof}

The endpoint is determined by this realization:
\begin{equation}
\label{eq:caputo-endpoint}
\CD f(0)
:=
\lim_{t\downarrow0}\CD f(t)
=
\rho a_1.
\end{equation}
All powers $\CD^m$ below denote sequential iterates.

Fractional integration gives a bounded right inverse.

\begin{proposition}[Right inverse and truncation identities]
\label{prop:inverse}
The operator $S=I_{0+}^{\alpha}$ acts boundedly on $\X$.
For $f=J_{\alpha,\rho}a$,
\begin{equation}
\label{eq:integration-coefficients}
Sf
=
\rho^{-1}\sum_{j=0}^{\infty}a_je_{j+1}.
\end{equation}
Consequently,
\begin{equation}
\label{eq:inverse}
\CD S=I,
\qquad
S\CD=I-P_1,
\end{equation}
where $P_1f=f(0)e_0$.

More generally, for every integer $m\ge1$,
\begin{equation}
\label{eq:inverse-powers}
\CD^mS^m=I,
\qquad
S^m\CD^m=I-P_m,
\qquad
\norm{S^m}=\rho^{-m}.
\end{equation}
In particular,
\begin{equation}
\label{eq:kernel-range}
\ker\CD^m
=
\spanop\{e_0,\ldots,e_{m-1}\},
\qquad
\operatorname{ran}\CD^m=\X.
\end{equation}
\end{proposition}

\begin{proof}
For each $n\ge0$, the fractional integral formula gives
\[
I_{0+}^{\alpha}e_n(t)
=
\frac{\rho^nt^{\alpha(n+1)}}
{\Gamma(\alpha(n+1)+1)}
=
\rho^{-1}e_{n+1}(t).
\]

Termwise integration is justified by
\[
\begin{aligned}
&\frac{1}{\Gamma(\alpha)}
\sum_{n=0}^{\infty}|a_n|
\int_0^t(t-s)^{\alpha-1}e_n(s)\,ds\\
&\qquad=
\rho^{-1}
\sum_{n=0}^{\infty}|a_n|e_{n+1}(t)
<\infty.
\end{aligned}
\]
This proves \eqref{eq:integration-coefficients}.

Let $F$ be the forward shift on $\ell^p$:
\[
F(a_0,a_1,\ldots)=(0,a_0,a_1,\ldots).
\]
Then
\[
J_{\alpha,\rho}^{-1}SJ_{\alpha,\rho}
=
\rho^{-1}F.
\]
Since $F$ is an isometry,
\[
\norm{S^mf}_{\X}
=
\rho^{-m}\norm{f}_{\X}.
\]

Writing $\pi_m$ for truncation to the first $m$
coordinates, the shift identities are
\[
B^mF^m=I,
\qquad
F^mB^m=I-\pi_m.
\]
Conjugation by $J_{\alpha,\rho}$ proves
\eqref{eq:inverse-powers}. The coefficient formula
for $\CD^m$ gives its kernel, while
$\CD^mS^m=I$ gives surjectivity.
\end{proof}

These identities provide the operators used in the
approximation arguments. The coefficient norm is
essential to their boundedness.

\begin{proposition}[Failure of a supremum-norm bound]
\label{prop:supremum-unbounded}
For every $R>0$, there is no constant $C_R>0$ such that
\[
\norm{{}^CD_{0+}^{\alpha}f}_{C[0,R]}
\le
C_R\norm{f}_{C[0,R]}
\qquad(f\in\Pol),
\]
where the Caputo derivative at zero is interpreted by
continuous extension.
\end{proposition}

\begin{proof}
For $n\ge1$, let
\[
u_n(t)=\left(\frac{t}{R}\right)^{\alpha n}.
\]
Then $u_n\in\Pol$ and
\[
\norm{u_n}_{C[0,R]}=1.
\]
The power formula yields
\[
{}^CD_{0+}^{\alpha}u_n(t)
=
R^{-\alpha n}
\frac{\Gamma(\alpha n+1)}
{\Gamma(\alpha(n-1)+1)}
t^{\alpha(n-1)}.
\]
Therefore,
\[
\begin{aligned}
\norm{{}^CD_{0+}^{\alpha}u_n}_{C[0,R]}
&=
R^{-\alpha}
\frac{\Gamma(\alpha n+1)}
{\Gamma(\alpha(n-1)+1)}\\
&\sim
R^{-\alpha}(\alpha n)^\alpha
\longrightarrow\infty.
\end{aligned}
\]
The asserted estimate is impossible.
\end{proof}

% ============================================================
\section{Local defects and optimal perturbations}
\label{sec:local-defects}
% ============================================================

Fix $\lambda\in\C\setminus\{0\}$ and write
\[
A=\lambda\CD,
\qquad
c=\lambda\rho,
\qquad
J=J_{\alpha,\rho},
\qquad
S_\lambda=\lambda^{-1}S.
\]
The coefficient representation gives, for $n\ge0$,
\begin{equation}
\label{eq:scaled-shift-identities}
A^nJ=c^nJB^n,
\qquad
S_\lambda^nJ=c^{-n}JF^n,
\end{equation}
where $B$ and $F$ are the backward and forward shifts
on $\ell^p(\Nzero)$. In particular,
\begin{equation}
\label{eq:scaled-inverse-identities}
A^nS_\lambda^n=I,
\qquad
S_\lambda^nA^n=I-P_n,
\qquad
\norm{A^n}=|c|^n,
\end{equation}
with $P_0=0$. Moreover,
\begin{equation}
\label{eq:scaled-inverse-norm}
\norm{S_\lambda^ng}_{\X}
=
|c|^{-n}\norm{g}_{\X}
\qquad(g\in\X).
\end{equation}

These identities allow both approximation problems
to be solved explicitly.

% ------------------------------------------------------------
\subsection{Balanced errors and exact arrival}
% ------------------------------------------------------------

For $f,g\in\X$, define
\begin{align}
\Delta_{A,n}(f,g)
&=
\inf_{h\in\X}
\max\left\{
\norm{h}_{\X},
\norm{A^n(f+h)-g}_{\X}
\right\},
\label{eq:defect-n}\\
\Delta_A(f,g)
&=
\inf_{n\ge0}\Delta_{A,n}(f,g).
\label{eq:defect}
\end{align}
The first quantity balances the initial displacement
against the final error.

\begin{theorem}[Exact balanced approximation]
\label{thm:local-cost}
Let $f,g\in\X$ and $n\ge0$. Then
\begin{equation}
\label{eq:local-cost}
\boxed{
\Delta_{A,n}(f,g)
=
\frac{\norm{g-A^nf}_{\X}}{1+|c|^n}.
}
\end{equation}
The infimum is attained at
\begin{equation}
\label{eq:balanced-h}
h_{n,*}
=
\frac{|c|^n}{1+|c|^n}
S_\lambda^n(g-A^nf).
\end{equation}
Both errors have the same value:
\begin{equation}
\label{eq:equal-errors}
\norm{h_{n,*}}_{\X}
=
\norm{A^n(f+h_{n,*})-g}_{\X}
=
\Delta_{A,n}(f,g).
\end{equation}

If $f=Ja$ and $g=Jb$, then
\begin{equation}
\label{eq:local-coefficients}
\Delta_{A,n}(f,g)
=
\frac{
\left(
\sum_{j=0}^{\infty}
|b_j-c^na_{n+j}|^p
\right)^{1/p}
}
{1+|c|^n}.
\end{equation}
\end{theorem}

\begin{proof}
Set
\[
w=g-A^nf,
\qquad
M=|c|^n.
\]
For every $h\in\X$,
\[
\begin{aligned}
\norm{w}_{\X}
&\le
\norm{w-A^nh}_{\X}
+
\norm{A^nh}_{\X}\\
&\le
\norm{w-A^nh}_{\X}
+
M\norm{h}_{\X}\\
&\le
(1+M)
\max\left\{
\norm{h}_{\X},
\norm{A^nh-w}_{\X}
\right\}.
\end{aligned}
\]
Therefore,
\[
\Delta_{A,n}(f,g)
\ge
\frac{\norm{w}_{\X}}{1+M}.
\]

For
\[
h_*=\frac{M}{1+M}S_\lambda^nw,
\]
equations \eqref{eq:scaled-inverse-identities}
and \eqref{eq:scaled-inverse-norm} yield
\[
\norm{h_*}_{\X}
=
\frac{\norm{w}_{\X}}{1+M},
\qquad
A^nh_*-w=-\frac{w}{1+M}.
\]
Thus the lower bound is attained and
\eqref{eq:equal-errors} follows.

Finally,
\[
w
=
\sum_{j=0}^{\infty}
(b_j-c^na_{n+j})e_j,
\]
which proves \eqref{eq:local-coefficients}.
\end{proof}

Unequal error tolerances admit an equally explicit test.

\begin{proposition}[Prescribed error budgets]
\label{prop:error-budgets}
Let $f,g\in\X$, $n\ge0$ and $\varepsilon,\eta\ge0$.
There exists $h\in\X$ such that
\begin{equation}
\label{eq:error-budgets}
\norm{h}_{\X}\le\varepsilon,
\qquad
\norm{A^n(f+h)-g}_{\X}\le\eta
\end{equation}
if and only if
\begin{equation}
\label{eq:error-feasibility}
\norm{g-A^nf}_{\X}
\le
|c|^n\varepsilon+\eta.
\end{equation}
\end{proposition}

\begin{proof}
Write $w=g-A^nf$ and $M=|c|^n$.
If \eqref{eq:error-budgets} holds, then
\[
\norm{w}_{\X}
\le
\norm{A^nh}_{\X}
+
\norm{w-A^nh}_{\X}
\le
M\varepsilon+\eta.
\]

Conversely, assume \eqref{eq:error-feasibility}.
If $w=0$, take $h=0$.
Otherwise, set
\[
\theta
=
\min\left\{
1,\frac{M\varepsilon}{\norm{w}_{\X}}
\right\},
\qquad
h=\theta S_\lambda^nw.
\]
Then
\[
\norm{h}_{\X}
=
\theta M^{-1}\norm{w}_{\X}
\le\varepsilon,
\]
and
\[
\begin{aligned}
\norm{A^nh-w}_{\X}
&=
(1-\theta)\norm{w}_{\X}\\
&=
\max\left\{
0,\norm{w}_{\X}-M\varepsilon
\right\}
\le\eta.
\end{aligned}
\]
\end{proof}

Setting the terminal error to zero leads to the
exact-arrival problem.

\begin{theorem}[Unique least-norm exact perturbation]
\label{thm:optimal}
Let $f=Ja$, $g=Jb$ belong to $\X$, and let $n\ge1$.
The constraint
\[
A^n(f+h)=g
\]
has a unique least-norm solution,
\begin{equation}
\label{eq:hoptimal}
h_n
=
S_\lambda^n(g-A^nf)
=
-(I-P_n)f+S_\lambda^ng.
\end{equation}
Its norm is
\begin{align}
\dist\bigl(f,A^{-n}\{g\}\bigr)
&=
\norm{h_n}_{\X}
\notag\\
&=
\left(
\sum_{j=0}^{\infty}
|c^{-n}b_j-a_{n+j}|^p
\right)^{1/p}
\label{eq:optimal}\\
&=
|c|^{-n}\norm{g-A^nf}_{\X}.
\notag
\end{align}
Consequently,
\begin{equation}
\label{eq:optimal-bound}
\norm{h_n}_{\X}
\le
\norm{(I-P_n)f}_{\X}
+
|c|^{-n}\norm{g}_{\X}.
\end{equation}
Here $A^{-n}\{g\}$ denotes the preimage of $\{g\}$.
\end{theorem}

\begin{proof}
The right-inverse identity gives
\[
A^nh_n=g-A^nf.
\]
Hence every admissible perturbation has the form
\[
h=h_n+k,
\qquad
k\in\ker A^n=P_n\X.
\]
The first $n$ coefficients of $h_n$ vanish,
whereas all coefficients of $k$ after index $n-1$
vanish. Therefore,
\[
\norm{h}_{\X}^p
=
\norm{h_n}_{\X}^p+\norm{k}_{\X}^p.
\]
Since $p<\infty$, equality
$\norm{h}_{\X}=\norm{h_n}_{\X}$ holds precisely
when $k=0$. This proves uniqueness.

In coefficients,
\[
h_n
=
\sum_{j=0}^{\infty}
(c^{-n}b_j-a_{n+j})e_{n+j}.
\]
Taking its norm gives \eqref{eq:optimal}.
The identity
$S_\lambda^nA^n=I-P_n$ and the triangle inequality
give \eqref{eq:optimal-bound}.
\end{proof}

% ------------------------------------------------------------
\subsection{Approximation thresholds and spectral data}
% ------------------------------------------------------------

The coefficient tails determine the asymptotic costs.

\begin{lemma}[Local defects and transitivity]
\label{lem:defect-characterization}
Let $X$ be a separable infinite-dimensional Banach space
and $T\in\mathcal L(X)$. For $u,v\in X$, set
\[
\Delta_{T,n}(u,v)
=
\inf_{h\in X}
\max\left\{
\norm{h},
\norm{T^n(u+h)-v}
\right\},
\qquad
\Delta_T(u,v)=\inf_{n\ge0}\Delta_{T,n}(u,v).
\]
Then:
\begin{enumerate}
\item $T$ is hypercyclic if and only if
\[
\Delta_T(u,v)=0
\qquad(u,v\in X);
\]
\item $T$ is mixing if and only if
\[
\Delta_{T,n}(u,v)\longrightarrow0
\qquad(u,v\in X).
\]
\end{enumerate}
\end{lemma}

\begin{proof}
For every $n\ge0$ and $\varepsilon>0$,
the definition of the infimum gives
\begin{equation}
\label{eq:defect-ball-equivalence}
\begin{aligned}
\Delta_{T,n}(u,v)<\varepsilon
&\quad\Longleftrightarrow\quad
\exists h\in X:
\ \norm{h}<\varepsilon,\
\norm{T^n(u+h)-v}<\varepsilon\\
&\quad\Longleftrightarrow\quad
T^n\bigl(B_X(u,\varepsilon)\bigr)
\cap B_X(v,\varepsilon)\ne\varnothing.
\end{aligned}
\end{equation}

Consequently, $\Delta_T(u,v)=0$ for every $u,v$
is equivalent to transitivity. Indeed, every pair
of nonempty open sets contains balls with a common
positive radius. The first assertion therefore
follows from Theorem~\ref{thm:birkhoff}.

Likewise, convergence of $\Delta_{T,n}(u,v)$ to zero
means that the intersection in
\eqref{eq:defect-ball-equivalence} is nonempty for
every sufficiently large $n$, for each choice of
$u,v$ and $\varepsilon$.
This is equivalent to mixing.
\end{proof}

\begin{corollary}[Sharp approximation threshold]
\label{cor:local-threshold}
\label{cor:mixing}
Let $A=\lambda\CD$, where $\lambda\ne0$, and set
$c=\lambda\rho$.

If $|c|>1$, then, for every $f,g\in\X$,
\begin{equation}
\label{eq:defect-tail-bound}
\begin{aligned}
\Delta_{A,n}(f,g)
&\le
\frac{\norm{g}_{\X}}{1+|c|^n}\\
&\quad+
\frac{|c|^n}{1+|c|^n}
\norm{(I-P_n)f}_{\X}
\longrightarrow0.
\end{aligned}
\end{equation}
Hence $A$ is mixing and hypercyclic.

Moreover, the vectors
\begin{equation}
\label{eq:exact-orbits}
z_n=P_nf+S_\lambda^ng
\end{equation}
satisfy
\[
z_n\longrightarrow f,
\qquad
A^nz_n=g
\qquad(n\ge1).
\]

If $0<|c|\le1$, then
\begin{equation}
\label{eq:subcritical-defect}
\Delta_A(0,g)=\frac12\norm{g}_{\X}
\qquad(g\in\X),
\end{equation}
and $A$ is not hypercyclic.
\end{corollary}

\begin{proof}
For $f=Ja$, the coefficient representation gives
\[
\norm{A^nf}_{\X}
=
|c|^n
\left(
\sum_{j=n}^{\infty}|a_j|^p
\right)^{1/p}
=
|c|^n\norm{(I-P_n)f}_{\X}.
\]
Theorem~\ref{thm:local-cost} and the triangle
inequality therefore yield
\[
\begin{aligned}
\Delta_{A,n}(f,g)
&=
\frac{\norm{g-A^nf}_{\X}}{1+|c|^n}\\
&\le
\frac{\norm{g}_{\X}}{1+|c|^n}
+
\frac{|c|^n}{1+|c|^n}
\norm{(I-P_n)f}_{\X}.
\end{aligned}
\]
Since $a\in\ell^p$,
\[
\norm{(I-P_n)f}_{\X}\longrightarrow0.
\]
When $|c|>1$, the first term also tends to zero,
whereas $|c|^n/(1+|c|^n)\le1$.
Thus \eqref{eq:defect-tail-bound} holds.
Lemma~\ref{lem:defect-characterization} gives
mixing, and Theorem~\ref{thm:birkhoff} gives
hypercyclicity.

For the vectors in \eqref{eq:exact-orbits},
\[
\norm{z_n-f}_{\X}
\le
\norm{(I-P_n)f}_{\X}
+
|c|^{-n}\norm{g}_{\X}
\longrightarrow0.
\]
The identities $A^nP_n=0$ and
$A^nS_\lambda^n=I$ give $A^nz_n=g$.

Finally, if $0<|c|\le1$, then
\[
\Delta_{A,n}(0,g)
=
\frac{\norm{g}_{\X}}{1+|c|^n}.
\]
The denominator is at most two, with equality
at $n=0$. Hence
\[
\Delta_A(0,g)=\frac12\norm{g}_{\X}.
\]
For $g\ne0$, this value is positive.
The first assertion of
Lemma~\ref{lem:defect-characterization}
therefore excludes hypercyclicity.
\end{proof}

The same representation determines the spectral scale.

\begin{proposition}[Spectrum, eigenvectors and resolvent]
\label{prop:spectrum}
The spectrum and point spectrum of $\CD$ are
\begin{equation}
\label{eq:caputo-spectrum}
\sigma(\CD)
=
\{\mu\in\C:|\mu|\le\rho\},
\qquad
\sigma_p(\CD)
=
\{\mu\in\C:|\mu|<\rho\}.
\end{equation}
For $|\mu|<\rho$,
\begin{equation}
\label{eq:caputo-eigenspace}
\ker(\CD-\mu I)
=
\spanop\{E_\alpha(\mu t^\alpha)\},
\end{equation}
and
\begin{equation}
\label{eq:eigenvector-norm}
\norm{E_\alpha(\mu t^\alpha)}_{\X}
=
\left(1-|\mu/\rho|^p\right)^{-1/p}.
\end{equation}

For $|\mu|>\rho$, the resolvent satisfies
\begin{equation}
\label{eq:caputo-resolvent}
(\mu I-\CD)^{-1}
=
\sum_{k=0}^{\infty}\mu^{-k-1}\CD^k
\end{equation}
in operator norm, with
\begin{equation}
\label{eq:resolvent-norm}
\norm{(\mu I-\CD)^{-1}}
=
\frac1{|\mu|-\rho}.
\end{equation}
\end{proposition}

\begin{proof}
For $f=Ja$, the equation $\CD f=\mu f$ is equivalent to
\[
\rho a_{n+1}=\mu a_n
\qquad(n\ge0).
\]
Thus
\[
a_n=(\mu/\rho)^na_0.
\]
A nonzero solution belongs to $\ell^p$ exactly when
$|\mu|<\rho$. Its function representation is
\[
a_0\sum_{n=0}^{\infty}
\frac{\mu^nt^{\alpha n}}{\Gamma(\alpha n+1)}
=
a_0E_\alpha(\mu t^\alpha).
\]
This proves \eqref{eq:caputo-eigenspace}, and geometric
summation gives \eqref{eq:eigenvector-norm}.

Since the spectrum is closed and contains the point
spectrum, it contains the closed disk of radius $\rho$.
The inclusion
\[
\sigma(\CD)\subset\{\mu:|\mu|\le\norm{\CD}\}
\]
and $\norm{\CD}=\rho$ prove
\eqref{eq:caputo-spectrum}.

If $|\mu|>\rho$, the series in
\eqref{eq:caputo-resolvent} converges because
\[
\sum_{k=0}^{\infty}
|\mu|^{-k-1}\norm{\CD^k}
=
\sum_{k=0}^{\infty}
|\mu|^{-k-1}\rho^k
=
\frac1{|\mu|-\rho}.
\]
Hence
\[
\norm{(\mu I-\CD)^{-1}}
\le
\frac1{|\mu|-\rho}.
\]

For $|\nu|<\rho$, let
$v_\nu=E_\alpha(\nu t^\alpha)$.
Then
\[
(\mu I-\CD)^{-1}v_\nu
=
(\mu-\nu)^{-1}v_\nu,
\]
so
\[
\norm{(\mu I-\CD)^{-1}}
\ge
\frac1{|\mu-\nu|}.
\]
Taking
\[
\nu=r\frac{\mu}{|\mu|},
\qquad r\uparrow\rho,
\]
gives the reverse inequality in
\eqref{eq:resolvent-norm}.
\end{proof}
% ============================================================
\section{Periodic approximation and the chaos threshold}
\label{sec:periodic-chaos}
% ============================================================

Write
\[
A=\lambda\CD,
\qquad
c=\lambda\rho.
\]
For the periodic constructions, assume $|c|>1$.
The final classification covers all $\lambda\in\C$.

% ------------------------------------------------------------
\subsection{Periodic subspaces and approximation}
% ------------------------------------------------------------

The equation $A^mf=f$ determines all coefficients
from one finite block.

\begin{theorem}[Periodic block representation]
\label{thm:periodic}
For every integer $m\ge1$, the map
\begin{equation}
\label{eq:periodic}
\Pi_m:\C^m\longrightarrow\ker(A^m-I),
\qquad
\Pi_md
=
\sum_{k=0}^{\infty}
\sum_{j=0}^{m-1}
c^{-km}d_je_{km+j},
\end{equation}
is a linear isomorphism. Moreover,
\begin{equation}
\label{eq:periodicnorm}
\norm{\Pi_md}_{\X}^p
=
\frac{\displaystyle\sum_{j=0}^{m-1}|d_j|^p}
{1-|c|^{-mp}}.
\end{equation}
In particular,
\begin{equation}
\label{eq:periodic-dimension}
\dim\ker(A^m-I)=m.
\end{equation}

If
\[
a=\sum_{j=0}^{L}d_je_j,
\qquad m>L,
\]
and $d$ is extended by zeros to length $m$, then
\begin{equation}
\label{eq:periodicerror}
\norm{\Pi_md-a}_{\X}
=
\frac{|c|^{-m}}
{(1-|c|^{-mp})^{1/p}}
\norm{a}_{\X}.
\end{equation}
\end{theorem}

\begin{proof}
For $f=\sum_{n\ge0}a_ne_n$, the coefficient formula
for $A^m$ gives
\[
A^mf=f
\quad\Longleftrightarrow\quad
c^ma_{n+m}=a_n
\qquad(n\ge0).
\]
Thus
\[
a_{km+j}=c^{-km}a_j,
\qquad
k\ge0,\quad 0\le j<m.
\]
Every periodic vector therefore has the form
\eqref{eq:periodic}, with a uniquely determined
initial block.

Conversely, for $d\in\C^m$,
\[
\begin{aligned}
\sum_{k=0}^{\infty}
\sum_{j=0}^{m-1}
|c^{-km}d_j|^p
&=
\left(\sum_{j=0}^{m-1}|d_j|^p\right)
\sum_{k=0}^{\infty}|c|^{-kmp}\\
&=
\frac{\displaystyle\sum_{j=0}^{m-1}|d_j|^p}
{1-|c|^{-mp}}
<\infty.
\end{aligned}
\]
Hence the series defines an element of $\X$.
Its coefficients satisfy the periodicity relation,
so its sum belongs to $\ker(A^m-I)$.
This proves the isomorphism and
\eqref{eq:periodicnorm}.

For the polynomial $a$, the block $k=0$ equals $a$.
The remaining blocks have disjoint coefficient
supports, giving
\[
\begin{aligned}
\norm{\Pi_md-a}_{\X}^p
&=
\sum_{k=1}^{\infty}|c|^{-kmp}
\sum_{j=0}^{L}|d_j|^p\\
&=
\frac{|c|^{-mp}}{1-|c|^{-mp}}
\norm{a}_{\X}^p.
\end{aligned}
\]
Taking the $p$th root proves
\eqref{eq:periodicerror}.
\end{proof}

Repeating the initial coefficient block produces
bounded periodic approximations of arbitrary vectors.

\begin{corollary}[Periodic approximation operators]
\label{cor:arbitrary-periodic}
Let
\[
P_mf=\sum_{j=0}^{m-1}a_je_j,
\qquad
f=\sum_{j=0}^{\infty}a_je_j.
\]
Define
\begin{equation}
\label{eq:periodic-projection}
\mathcal Q_mf
=
\sum_{k=0}^{\infty}
\sum_{j=0}^{m-1}
c^{-km}a_je_{km+j}.
\end{equation}
Then
\begin{equation}
\label{eq:periodic-projection-identities}
\mathcal Q_m^2=\mathcal Q_m,
\qquad
\operatorname{ran}\mathcal Q_m=\ker(A^m-I),
\end{equation}
and
\begin{equation}
\label{eq:periodic-projection-norm}
\norm{\mathcal Q_m}
=
(1-|c|^{-mp})^{-1/p}.
\end{equation}
Furthermore,
\begin{equation}
\label{eq:period-general}
\begin{aligned}
\norm{\mathcal Q_mf-f}_{\X}
&\le
\norm{(I-P_m)f}_{\X}\\
&\quad+
\frac{|c|^{-m}}
{(1-|c|^{-mp})^{1/p}}
\norm{P_mf}_{\X}
\longrightarrow0.
\end{aligned}
\end{equation}
Consequently,
\[
\overline{\Per(A)}=\X.
\]
\end{corollary}

\begin{proof}
Theorem~\ref{thm:periodic} gives
\[
\mathcal Q_mf
=
\Pi_m(a_0,\ldots,a_{m-1})
\in\ker(A^m-I).
\]
Since $\mathcal Q_mf$ has the same initial block
as $f$,
\[
P_m\mathcal Q_mf=P_mf.
\]
A vector in $\ker(A^m-I)$ is determined by that block.
Hence $\mathcal Q_m$ fixes its range, proving
\eqref{eq:periodic-projection-identities}.

Equation \eqref{eq:periodicnorm} yields
\[
\norm{\mathcal Q_mf}_{\X}
=
(1-|c|^{-mp})^{-1/p}\norm{P_mf}_{\X}.
\]
Since $\norm{P_m}=1$, this proves the upper bound
in \eqref{eq:periodic-projection-norm}.
Equality holds for any unit vector in $P_m\X$.

The initial blocks of $\mathcal Q_mf$ and $f$
cancel. Thus the error also has the exact expression
\begin{equation}
\label{eq:periodic-error-coefficients}
\norm{\mathcal Q_mf-f}_{\X}^p
=
\sum_{k=1}^{\infty}
\sum_{j=0}^{m-1}
\left|
c^{-km}a_j-a_{km+j}
\right|^p.
\end{equation}
Alternatively, writing
\[
\mathcal Q_mf-f
=
(\mathcal Q_mf-P_mf)-(I-P_m)f
\]
and summing the repeated blocks gives
\eqref{eq:period-general}.

The coefficient tail tends to zero because
$a\in\ell^p$, while the geometric factor tends
to zero because $|c|>1$.
Therefore $\mathcal Q_mf\to f$.
Each $\mathcal Q_mf$ is periodic, proving density.
\end{proof}

% ------------------------------------------------------------
\subsection{Explicit periods and dynamical classification}
% ------------------------------------------------------------

The periodic blocks admit closed Mittag--Leffler
representations.

\begin{proposition}[Periodic bases and eigenvectors]
\label{prop:periodforms}
For $m\ge1$ and $0\le j<m$, set
\begin{equation}
\label{eq:periodbasis}
q_{m,j}(t)
=
\rho^jt^{\alpha j}
E_{\alpha m,\alpha j+1}
\bigl(\lambda^{-m}t^{\alpha m}\bigr).
\end{equation}
Then $(q_{m,j})_{j=0}^{m-1}$ is a basis of
$\ker(A^m-I)$.

Let
\[
\omega_r=e^{2\pi ir/m},
\qquad
v_{m,r}(t)
=
E_\alpha(\lambda^{-1}\omega_rt^\alpha),
\qquad 0\le r<m.
\]
Then
\begin{equation}
\label{eq:periodic-eigenvectors}
Av_{m,r}=\omega_rv_{m,r},
\end{equation}
and
\begin{equation}
\label{eq:periodic-spectral-decomposition}
\ker(A^m-I)
=
\bigoplus_{r=0}^{m-1}
\spanop\{v_{m,r}\}.
\end{equation}
The two bases are related by
\begin{equation}
\label{eq:periodic-fourier}
q_{m,j}
=
\frac{c^j}{m}
\sum_{r=0}^{m-1}\omega_r^{-j}v_{m,r}.
\end{equation}

In particular,
\begin{equation}
\label{eq:exactperiod}
g_m(t)
=
E_{\alpha m}
\bigl(\lambda^{-m}t^{\alpha m}\bigr)
\end{equation}
has minimal period $m$, and
\begin{equation}
\label{eq:rootfilter}
g_m(t)
=
\frac1m\sum_{r=0}^{m-1}
E_\alpha(\lambda^{-1}\omega_rt^\alpha).
\end{equation}
\end{proposition}

\begin{proof}
The periodic vector whose initial block is the
$j$th coordinate vector is
\[
\sum_{k=0}^{\infty}c^{-km}e_{km+j}(t).
\]
Substituting the definition of $e_n$ gives
\[
\begin{aligned}
\sum_{k=0}^{\infty}c^{-km}e_{km+j}(t)
&=
\rho^jt^{\alpha j}
\sum_{k=0}^{\infty}
\frac{\lambda^{-km}t^{\alpha km}}
{\Gamma(\alpha km+\alpha j+1)}\\
&=
q_{m,j}(t).
\end{aligned}
\]
Theorem~\ref{thm:periodic} therefore proves the
basis assertion.

Since
\[
\frac{|\lambda^{-1}\omega_r|}{\rho}
=
|c|^{-1}<1,
\]
each $v_{m,r}$ belongs to $\X$.
The eigenvector formula for $\CD$ yields
\eqref{eq:periodic-eigenvectors}.

The root-of-unity identity
\[
\frac1m\sum_{r=0}^{m-1}\omega_r^{n-j}
=
\begin{cases}
1,&n\equiv j\pmod m,\\
0,&n\not\equiv j\pmod m
\end{cases}
\]
gives
\[
\begin{aligned}
\frac{c^j}{m}
\sum_{r=0}^{m-1}\omega_r^{-j}v_{m,r}(t)
&=
c^j\sum_{k=0}^{\infty}
\frac{\lambda^{-(km+j)}t^{\alpha(km+j)}}
{\Gamma(\alpha(km+j)+1)}\\
&=
q_{m,j}(t).
\end{aligned}
\]
This proves \eqref{eq:periodic-fourier}.
The Fourier matrix is invertible, so
$(v_{m,r})_{r=0}^{m-1}$ is also a basis.
Its members have distinct eigenvalues, giving
\eqref{eq:periodic-spectral-decomposition}.

Taking $j=0$ proves \eqref{eq:rootfilter}.
The initial coefficients of $g_m=q_{m,0}$ are
\[
a_0=1,
\qquad
a_1=\cdots=a_{m-1}=0.
\]
For $1\le q<m$, the coefficient of $e_0$ in
$A^qg_m$ is $c^qa_q=0$, whereas that coefficient
in $g_m$ is one. Thus $A^qg_m\ne g_m$.
Since $A^mg_m=g_m$, its minimal period is $m$.
\end{proof}

The approximation can also prescribe every sufficiently
large minimal period.

\begin{corollary}[Approximation with large minimal periods]
\label{cor:large-minimal-periods}
Let
\[
a=\sum_{j=0}^{L}d_je_j,
\qquad
d_L\ne0.
\]
For every integer $m>2L$, the periodic vector
$\Pi_m(d_0,\ldots,d_L,0,\ldots,0)$ has minimal
period $m$.

Consequently, for every nonempty open set $U\subset\X$,
there exists $m_0\ge1$ such that $U$ contains a point
of minimal period $m$ for every $m\ge m_0$.
\end{corollary}

\begin{proof}
Let $f_m=\Pi_m(d_0,\ldots,d_L,0,\ldots,0)$,
and denote its minimal period by $q$.
Since $m$ is a period, division of $m$ by $q$
shows that $q$ divides $m$.

If $q<m$, then $q\le m/2$.
Because $L<m/2$, we have
\[
L<L+q<m.
\]
The initial block of $f_m$ therefore satisfies
\[
a_L=d_L\ne0,
\qquad
a_{L+q}=0.
\]
On the other hand, $A^qf_m=f_m$ requires
\[
a_{L+q}=c^{-q}a_L\ne0,
\]
a contradiction. Thus $q=m$.

Now choose a nonzero fractional polynomial $a\in U$.
By \eqref{eq:periodicerror}, $f_m\to a$ as $m\to\infty$.
For all sufficiently large $m$, one has both
$f_m\in U$ and $m>2L$.
\end{proof}

It remains to verify sensitivity and identify the
opposite parameter regime.

\begin{proposition}[Sensitivity]
\label{prop:sensitive}
If $|c|>1$, then $A$ is sensitive with sensitivity
constant $\delta=1$.
\end{proposition}

\begin{proof}
Fix $f\in\X$ and $\varepsilon>0$.
Choose $n$ such that
\[
\frac{\varepsilon}{2}|c|^n>1,
\]
and set
\[
g=f+\frac{\varepsilon}{2}e_n.
\]
Then
\[
\norm{g-f}_{\X}
=
\frac{\varepsilon}{2}<\varepsilon,
\]
whereas
\[
\norm{A^ng-A^nf}_{\X}
=
\frac{\varepsilon}{2}|c|^n>1.
\]
The constant is independent of $f$ and
$\varepsilon$.
\end{proof}

\begin{theorem}[Sharp dynamical classification]
\label{thm:classification}
For $\lambda\in\C$, the following assertions
are equivalent on $\X$:
\begin{enumerate}
\item $|\lambda|\rho>1$;
\item $\lambda\CD$ is topologically mixing;
\item $\lambda\CD$ is hypercyclic;
\item $\lambda\CD$ is Devaney chaotic.
\end{enumerate}
If $|\lambda|\rho\le1$, then
\begin{equation}
\label{eq:subcritical-stability}
(\lambda\CD)^nf\longrightarrow0
\qquad(f\in\X),
\end{equation}
and
\begin{equation}
\label{eq:subcritical-periodic}
\Per(\lambda\CD)=\{0\}.
\end{equation}
\end{theorem}

\begin{proof}
For $|\lambda|\rho>1$,
Corollary~\ref{cor:local-threshold} gives mixing
and hypercyclicity.
Corollary~\ref{cor:arbitrary-periodic} gives
dense periodic points, and
Proposition~\ref{prop:sensitive} gives sensitivity.
Hence \textup{(i)} implies \textup{(ii)},
\textup{(iii)} and \textup{(iv)}.

Mixing implies transitivity, and Birkhoff's theorem
gives \textup{(ii)}$\Rightarrow$\textup{(iii)}.
Likewise, Devaney chaos includes transitivity,
so \textup{(iv)} implies \textup{(iii)}.

Suppose $0<|c|\le1$ and write $f=Ja$.
The coefficient formula gives
\[
\begin{aligned}
\norm{A^nf}_{\X}^p
&=
|c|^{np}\sum_{j=n}^{\infty}|a_j|^p\\
&\le
\sum_{j=n}^{\infty}|a_j|^p
\longrightarrow0.
\end{aligned}
\]
If $\lambda=0$, then $A^nf=0$ for $n\ge1$.
Thus \eqref{eq:subcritical-stability} holds
throughout the regime $|\lambda|\rho\le1$.

Every orbit is then bounded and cannot be dense
in the unbounded space $\X$.
This proves \textup{(iii)}$\Rightarrow$\textup{(i)}.

Finally, if $A^mf=f$ for some $m\ge1$, then
\[
A^{km}f=f
\qquad(k\ge1).
\]
Under \eqref{eq:subcritical-stability}, the left-hand
side tends to zero. Hence $f=0$, proving
\eqref{eq:subcritical-periodic}.
\end{proof}

% ============================================================
\section{Disjoint mixing of powers}
\label{sec:disjoint-powers}
% ============================================================

Let
\[
A=\lambda\CD,
\qquad
c=\lambda\rho,
\qquad
1\le r_1<\cdots<r_N,
\qquad N\ge2.
\]
Distinct powers can be controlled simultaneously
by placing the target coefficients in disjoint blocks.

\begin{proposition}[Simultaneous approximation estimates]
\label{prop:simultaneous-approximation}
Assume $|c|>1$ and set
\[
d=\min\{r_1,r_2-r_1,\ldots,r_N-r_{N-1}\},
\qquad
L_n=dn.
\]
For $f,g_1,\ldots,g_N\in\X$ and $n\ge1$, define
\begin{equation}
\label{eq:simultaneous-vector}
z_n
=
P_{L_n}f
+
\sum_{j=1}^{N}
S_\lambda^{r_jn}P_{L_n}g_j.
\end{equation}
Then
\begin{equation}
\label{eq:simultaneous-images}
A^{r_in}z_n
=
P_{L_n}g_i
+
\sum_{j=i+1}^{N}
S_\lambda^{(r_j-r_i)n}P_{L_n}g_j,
\qquad 1\le i\le N.
\end{equation}

The coefficient blocks satisfy
\begin{equation}
\label{eq:simultaneous-initial-cost}
\norm{z_n-P_{L_n}f}_{\X}^p
=
\sum_{j=1}^{N}
|c|^{-r_jnp}
\norm{P_{L_n}g_j}_{\X}^p
\end{equation}
and
\begin{equation}
\label{eq:simultaneous-target-cost}
\norm{A^{r_in}z_n-P_{L_n}g_i}_{\X}^p
=
\sum_{j=i+1}^{N}
|c|^{-(r_j-r_i)np}
\norm{P_{L_n}g_j}_{\X}^p.
\end{equation}
Consequently,
\[
z_n\longrightarrow f,
\qquad
A^{r_in}z_n\longrightarrow g_i
\quad(1\le i\le N).
\]
\end{proposition}

\begin{proof}
The coefficients of $P_{L_n}f$ are supported in
$\{0,\ldots,L_n-1\}$.
Those of
$S_\lambda^{r_jn}P_{L_n}g_j$
are supported in
\[
\{r_jn,\ldots,r_jn+L_n-1\}.
\]
These intervals are pairwise disjoint because
\[
r_1n\ge L_n,
\qquad
(r_{j+1}-r_j)n\ge L_n.
\]
Using
\[
\norm{S_\lambda^mu}_{\X}
=
|c|^{-m}\norm{u}_{\X},
\]
the disjoint supports give
\eqref{eq:simultaneous-initial-cost}.

The right-inverse identities imply
\[
A^uS_\lambda^v
=
\begin{cases}
A^{u-v},&u\ge v,\\
S_\lambda^{v-u},&u<v.
\end{cases}
\]
Moreover,
\[
A^qP_{L_n}=0
\qquad(q\ge L_n).
\]
Thus
\[
A^{r_in}P_{L_n}f=0.
\]
For $j<i$,
\[
(r_i-r_j)n\ge L_n,
\]
so
\[
A^{r_in}S_\lambda^{r_jn}P_{L_n}g_j
=
A^{(r_i-r_j)n}P_{L_n}g_j
=
0.
\]
The term $j=i$ equals $P_{L_n}g_i$, and the
remaining terms give
\eqref{eq:simultaneous-images}.

For fixed $i$, the remaining coefficient blocks
start at $(r_j-r_i)n$, $j>i$, and are pairwise
disjoint. This proves
\eqref{eq:simultaneous-target-cost}.

Finally,
\[
\begin{aligned}
\norm{z_n-f}_{\X}
&\le
\norm{(I-P_{L_n})f}_{\X}\\
&\quad+
\left(
\sum_{j=1}^{N}
|c|^{-r_jnp}\norm{P_{L_n}g_j}_{\X}^p
\right)^{1/p},
\end{aligned}
\]
and
\[
\begin{aligned}
\norm{A^{r_in}z_n-g_i}_{\X}
&\le
\norm{(I-P_{L_n})g_i}_{\X}\\
&\quad+
\left(
\sum_{j=i+1}^{N}
|c|^{-(r_j-r_i)np}
\norm{P_{L_n}g_j}_{\X}^p
\right)^{1/p}.
\end{aligned}
\]
The coefficient tails vanish as $L_n\to\infty$.
The remaining terms tend to zero because $|c|>1$.
\end{proof}

These estimates yield a quantitative version of
simultaneous local approximation.

\begin{corollary}[Simultaneous local defect]
\label{cor:disjoint-defect}
Assume $|c|>1$. For $f,g_1,\ldots,g_N\in\X$, define
\[
\begin{aligned}
\mathfrak D_n(f;g_1,\ldots,g_N)
=
\inf_{h\in\X}
\max\Bigl\{
\norm{h}_{\X},
\max_{1\le i\le N}
\norm{A^{r_in}(f+h)-g_i}_{\X}
\Bigr\}.
\end{aligned}
\]
Set
\[
G=
\left(
\sum_{j=1}^{N}\norm{g_j}_{\X}^p
\right)^{1/p},
\]
and
\[
\tau_L
=
\max\left\{
\norm{(I-P_L)f}_{\X},
\max_{1\le i\le N}
\norm{(I-P_L)g_i}_{\X}
\right\}.
\]
Then
\begin{equation}
\label{eq:disjoint-defect-bound}
\mathfrak D_n(f;g_1,\ldots,g_N)
\le
\tau_{L_n}+|c|^{-L_n}G
\longrightarrow0.
\end{equation}
\end{corollary}

\begin{proof}
Use $h=z_n-f$, with $z_n$ from
\eqref{eq:simultaneous-vector}.
Since
\[
r_jn\ge L_n,
\qquad
(r_j-r_i)n\ge L_n
\quad(j>i),
\]
each geometric term in the estimates of
Proposition~\ref{prop:simultaneous-approximation}
is bounded by $|c|^{-L_n}G$.
Taking the maximum gives
\eqref{eq:disjoint-defect-bound}.
\end{proof}

We can now identify the exact parameter range.

\begin{theorem}[Sharp disjoint dynamical threshold]
\label{thm:disjoint}
For $\lambda\in\C$, the following assertions
are equivalent:
\begin{enumerate}
\item $|\lambda|\rho>1$;
\item $(A^{r_1},\ldots,A^{r_N})$ is disjoint mixing;
\item $(A^{r_1},\ldots,A^{r_N})$ is disjoint hypercyclic.
\end{enumerate}
When these conditions hold, the disjoint hypercyclic
vectors form a dense $G_\delta$ subset of $\X$.
\end{theorem}

\begin{proof}
Assume \textup{(i)}.
Let $U_0,\ldots,U_N$ be nonempty open subsets of $\X$.
Choose
\[
f\in U_0,
\qquad
g_i\in U_i
\quad(1\le i\le N).
\]
By Proposition~\ref{prop:simultaneous-approximation},
\[
z_n\longrightarrow f,
\qquad
A^{r_in}z_n\longrightarrow g_i.
\]
Since only finitely many target sets are involved,
there exists $n_0$ such that
\[
z_n\in U_0,
\qquad
A^{r_in}z_n\in U_i
\quad(1\le i\le N,\ n\ge n_0).
\]
Hence
\[
U_0\cap
\bigcap_{i=1}^{N}
(A^{r_i})^{-n}(U_i)
\ne\varnothing
\qquad(n\ge n_0).
\]
This proves \textup{(ii)}.

Assume \textup{(ii)} and define
\[
\Phi_n:\X\longrightarrow {\X}^{N},
\qquad
\Phi_nx=(A^{r_1n}x,\ldots,A^{r_Nn}x).
\]
Let $(W_\ell)_{\ell\ge1}$ be a countable base
of nonempty open subsets of ${\X}^{N}$. Set
\[
H_\ell
=
\bigcup_{n\ge0}\Phi_n^{-1}(W_\ell).
\]
Each $H_\ell$ is open.

To prove density, take a nonempty open set
$U\subset\X$. Choose nonempty open sets
$V_1,\ldots,V_N\subset\X$ such that
\[
V_1\times\cdots\times V_N\subset W_\ell.
\]
Disjoint mixing gives, for all sufficiently large $n$,
\[
U\cap\bigcap_{i=1}^{N}
(A^{r_i})^{-n}(V_i)
\ne\varnothing.
\]
Thus $U\cap H_\ell\ne\varnothing$.

By Baire's theorem,
\[
H=\bigcap_{\ell\ge1}H_\ell
\]
is a dense $G_\delta$ subset of $\X$.
Moreover,
\[
x\in H
\quad\Longleftrightarrow\quad
\overline{\{\Phi_nx:n\ge0\}}={\X}^N.
\]
Therefore $H$ is precisely the set of disjoint
hypercyclic vectors. This proves \textup{(iii)}
and the residuality assertion.

Finally, suppose \textup{(iii)} holds.
The first coordinate projection shows that
$A^{r_1}$ is hypercyclic.

If $|\lambda|\rho\le1$, the stability conclusion
of Theorem~\ref{thm:classification} gives
\[
A^{r_1n}x\longrightarrow0
\qquad(x\in\X).
\]
No such orbit is dense in $\X$.
Hence $|\lambda|\rho>1$, proving
\textup{(iii)}$\Rightarrow$\textup{(i)}.
\end{proof}
% ============================================================
\section{Banach scales and comparison with a
Fr\'echet realization}
\label{sec:banach-scales}
% ============================================================

The radius parameter determines both the coefficient norm
and the dynamical threshold. We first compare the Banach
realizations, then embed them into the
Mittag--Leffler--Caputo space.

% ------------------------------------------------------------
\subsection{Compact inclusions and optimal truncations}
% ------------------------------------------------------------

\begin{proposition}[Compact Banach scale]
\label{prop:scale}
Fix $\alpha\in(0,1)$, $1\le p<\infty$ and
$0<\rho_1<\rho_2$. Set $\theta=\rho_1/\rho_2$.
The natural inclusion
\[
\iota:
X_{\alpha,\rho_1}^{p}
\longrightarrow
X_{\alpha,\rho_2}^{p}
\]
is injective, compact, and has dense proper range.
Moreover,
\begin{equation}
\label{eq:scale-norm}
\norm{\iota}=1,
\qquad
\norm{\iota-\iota P_N}=\theta^N
\qquad(N\ge0).
\end{equation}

These truncations are optimal:
\begin{equation}
\label{eq:optimal-scale-approximation}
\inf_{\operatorname{rank}R\le N}
\norm{\iota-R}
=
\theta^N,
\end{equation}
where
$R:X_{\alpha,\rho_1}^{p}\to X_{\alpha,\rho_2}^{p}$
ranges over bounded linear operators.
\end{proposition}

\begin{proof}
Write $J_i=J_{\alpha,\rho_i}$.
Uniqueness of coefficients gives
\[
J_2^{-1}\iota J_1(a_n)_{n\ge0}
=
(\theta^na_n)_{n\ge0}.
\]
Thus the inclusion is represented by the diagonal
operator
\[
D_\theta(a_n)=(\theta^na_n).
\]
Its norm is one, attained at the zeroth coordinate.
Likewise,
\[
\norm{D_\theta-D_\theta\pi_N}
=
\sup_{n\ge N}\theta^n
=
\theta^N,
\]
where $\pi_N$ truncates to the first $N$ coordinates.
This proves \eqref{eq:scale-norm} and compactness.

For optimality, let $R$ have rank at most $N$ and
consider the $(N+1)$-dimensional space
\[
E_N=\spanop\{e_0,\ldots,e_N\}
\subset X_{\alpha,\rho_1}^{p}.
\]
There exists $u\in E_N\cap\ker R$ with
$\norm{u}_{X_{\alpha,\rho_1}^{p}}=1$.
Since $\theta^j\ge\theta^N$ for $0\le j\le N$,
\[
\norm{(\iota-R)u}_{X_{\alpha,\rho_2}^{p}}
=
\norm{\iota u}_{X_{\alpha,\rho_2}^{p}}
\ge\theta^N.
\]
Together with the truncation upper bound, this proves
\eqref{eq:optimal-scale-approximation}.

Fractional polynomials belong to the range and are
dense in the target space.
To prove properness, choose $\mu$ with
\[
\rho_1\le|\mu|<\rho_2.
\]
The coefficient sequence of $E_\alpha(\mu t^\alpha)$
in the $\rho_i$ basis is $((\mu/\rho_i)^n)_{n\ge0}$.
Hence
\[
E_\alpha(\mu t^\alpha)
\in
X_{\alpha,\rho_2}^{p}
\setminus X_{\alpha,\rho_1}^{p}.
\]
\end{proof}

% ------------------------------------------------------------
\subsection{Embedding into the Mittag--Leffler--Caputo space}
% ------------------------------------------------------------

Let
\[
\Omega=\C\setminus(-\infty,0]
\]
and use the principal branch of $z^\alpha$.
Following \cite[Definition 10]{MPV}, consider
\[
\F
=
\left\{
f(z)=\sum_{n=0}^{\infty}b_nz^{\alpha n}:
Q_k(f)<\infty\text{ for every }k\ge0
\right\},
\]
where
\begin{equation}
\label{eq:qk}
Q_k(f)^2
=
\sum_{j=0}^{\infty}
\frac{\Gamma(\alpha(j+k)+1)^2}
{\Gamma(\alpha j+1)}
|b_{j+k}|^2.
\end{equation}
Introduce the normalized coefficients
\[
d_n=\Gamma(\alpha n+1)b_n.
\]
Then
\begin{equation}
\label{eq:normalized-graph-seminorms}
Q_k(f)^2
=
\sum_{n=k}^{\infty}
\frac{|d_n|^2}{\Gamma(\alpha(n-k)+1)}.
\end{equation}

We first specify the topology and the analytic meaning
of these series.

\begin{lemma}[Graph norms and local regularity]
\label{lem:graph}
The space $\F$ is a separable Fr\'echet space for the
increasing norms
\begin{equation}
\label{eq:graph-norms}
\mathfrak P_k(f)
=
\left(
\sum_{i=0}^{k}Q_i(f)^2
\right)^{1/2},
\qquad k\ge0.
\end{equation}
The same topology is generated by
\begin{equation}
\label{eq:alternative-graph-norms}
\widetilde{\mathfrak P}_k(f)
=
\left(
\sum_{n=0}^{k-1}|d_n|^2+Q_k(f)^2
\right)^{1/2},
\end{equation}
where the initial sum is zero when $k=0$.
This topology agrees with that of
\cite[Definition 10]{MPV}.

Every $f\in\F$ is holomorphic on $\Omega$, and
\begin{equation}
\label{eq:frechet-local-bound}
\sup_{\substack{z\in\Omega\\|z|\le R}}|f(z)|
\le
Q_0(f)\,E_\alpha(R^{2\alpha})^{1/2}
\qquad(R>0).
\end{equation}
Its restriction to the positive axis extends to
an element of $AC[0,R]$ with $f(0)=b_0$.
\end{lemma}

\begin{proof}
Recall that
\[
d_n=\Gamma(\alpha n+1)b_n,
\qquad
Q_i(f)^2
=
\sum_{n=i}^{\infty}
\frac{|d_n|^2}{\Gamma(\alpha(n-i)+1)}.
\]

For $k=0$, the two norms coincide:
\[
\widetilde{\mathfrak P}_0(f)
=
\mathfrak P_0(f)
=
Q_0(f).
\]
Now fix $k\ge1$ and set
\[
M_k=\max_{0\le n<k}\Gamma(\alpha n+1).
\]
Then
\[
\begin{aligned}
\sum_{n=0}^{k-1}|d_n|^2
&=
\sum_{n=0}^{k-1}
\Gamma(\alpha n+1)
\frac{|d_n|^2}{\Gamma(\alpha n+1)}\\
&\le
M_kQ_0(f)^2.
\end{aligned}
\]
Consequently,
\begin{equation}
\label{eq:graph-forward-estimate}
\widetilde{\mathfrak P}_k(f)^2
\le
M_kQ_0(f)^2+Q_k(f)^2
\le
(M_k+1)\mathfrak P_k(f)^2.
\end{equation}

For the reverse estimate, let $0\le i\le k$.
Stirling's formula gives
\[
\frac{\Gamma(\alpha(n-k)+1)}
{\Gamma(\alpha(n-i)+1)}
\sim
(\alpha n)^{-\alpha(k-i)}
\qquad(n\to\infty).
\]
Therefore,
\[
M_{i,k}
:=
\sup_{n\ge k}
\frac{\Gamma(\alpha(n-k)+1)}
{\Gamma(\alpha(n-i)+1)}
<\infty.
\]
The initial finite sum satisfies
\[
\sum_{n=i}^{k-1}
\frac{|d_n|^2}{\Gamma(\alpha(n-i)+1)}
\le
C_{i,k}^{(0)}
\sum_{n=0}^{k-1}|d_n|^2,
\]
where one may take
\[
C_{i,k}^{(0)}
=
\max_{0\le r\le k}
\frac1{\Gamma(\alpha r+1)}.
\]
This definition also covers $i=k$, for which the
sum on the left is empty.
For the remaining coefficients,
\[
\begin{aligned}
\sum_{n=k}^{\infty}
\frac{|d_n|^2}{\Gamma(\alpha(n-i)+1)}
&\le
M_{i,k}
\sum_{n=k}^{\infty}
\frac{|d_n|^2}{\Gamma(\alpha(n-k)+1)}\\
&=
M_{i,k}Q_k(f)^2.
\end{aligned}
\]
Thus
\[
Q_i(f)^2
\le
C_{i,k}\widetilde{\mathfrak P}_k(f)^2
\]
for a finite constant $C_{i,k}$.
Summing over $i$ yields
\begin{equation}
\label{eq:graph-reverse-estimate}
\mathfrak P_k(f)^2
\le
\left(\sum_{i=0}^{k}C_{i,k}\right)
\widetilde{\mathfrak P}_k(f)^2.
\end{equation}
Equations \eqref{eq:graph-forward-estimate} and
\eqref{eq:graph-reverse-estimate} prove equivalence
of the two families.

The defining norms in \cite[Definition 10]{MPV}
combine $Q_k$ with positive weights on finitely
many initial coefficients. For each fixed $k$,
these coordinate contributions are equivalent
to those in \eqref{eq:alternative-graph-norms}.
Hence the induced topologies agree.

We next verify that the coefficient conditions
define the asserted functions.
Cauchy--Schwarz gives
\begin{equation}
\label{eq:graph-series-bound}
\begin{aligned}
\sum_{n=0}^{\infty}|b_n|R^{\alpha n}
&=
\sum_{n=0}^{\infty}
\frac{|d_n|}{\sqrt{\Gamma(\alpha n+1)}}
\frac{R^{\alpha n}}{\sqrt{\Gamma(\alpha n+1)}}\\
&\le
Q_0(f)
\left(
\sum_{n=0}^{\infty}
\frac{R^{2\alpha n}}{\Gamma(\alpha n+1)}
\right)^{1/2}\\
&=
Q_0(f)\,E_\alpha(R^{2\alpha})^{1/2}.
\end{aligned}
\end{equation}
The series therefore converges absolutely and
uniformly on bounded subsets of $\Omega$.
Each summand is holomorphic there, proving
holomorphy and \eqref{eq:frechet-local-bound}.

To prove completeness, let $(f_\ell)$ be Cauchy
in every $\mathfrak P_k$ and write
$d_n^{(\ell)}$ for its normalized coefficients.
For every $n$,
\[
|d_n^{(\ell)}-d_n^{(s)}|
\le
\sqrt{\Gamma(\alpha n+1)}\,Q_0(f_\ell-f_s).
\]
Hence $d_n^{(\ell)}\to d_n$ for some $d_n\in\C$.

For fixed $\ell$ and $k$, Fatou's lemma gives
\[
\sum_{n=k}^{\infty}
\frac{|d_n^{(\ell)}-d_n|^2}
{\Gamma(\alpha(n-k)+1)}
\le
\liminf_{s\to\infty}Q_k(f_\ell-f_s)^2.
\]
The Cauchy property shows that the left-hand side
tends to zero as $\ell\to\infty$.
The weighted triangle inequality also gives
\[
\sum_{n=k}^{\infty}
\frac{|d_n|^2}{\Gamma(\alpha(n-k)+1)}
<\infty
\qquad(k\ge0).
\]
Thus the limit coefficients define an element
\[
f(z)=
\sum_{n=0}^{\infty}
\frac{d_n}{\Gamma(\alpha n+1)}z^{\alpha n}
\in\F,
\]
and $f_\ell\to f$ in every $\mathfrak P_k$.
Since these norms form a countable separating
family, $\F$ is a Fr\'echet space.

For each fixed $k$, coefficient truncation satisfies
\[
Q_k(f-P_Nf)^2
=
\sum_{n=\max\{N,k\}}^{\infty}
\frac{|d_n|^2}{\Gamma(\alpha(n-k)+1)}
\longrightarrow0.
\]
Fractional polynomials are therefore dense.
Approximating their coefficients by elements of
$\mathbb Q+i\mathbb Q$ produces a countable dense
subset, proving separability.

Finally, \eqref{eq:graph-series-bound} implies
\[
\begin{aligned}
\sum_{n=1}^{\infty}
\norm{b_n\alpha n\,t^{\alpha n-1}}_{L^1(0,R)}
&=
\sum_{n=1}^{\infty}|b_n|R^{\alpha n}\\
&<\infty.
\end{aligned}
\]
The derivative series converges in $L^1(0,R)$.
Integrating it term by term gives
\[
f(t)
=
b_0+
\int_0^t
\left(
\sum_{n=1}^{\infty}
\alpha n\,b_ns^{\alpha n-1}
\right)\,ds.
\]
Hence the extension with $f(0)=b_0$ belongs
to $AC[0,R]$.
\end{proof}
The graph seminorms also describe differentiation directly.

\begin{proposition}[Caputo action on the comparison space]
\label{prop:frechet-caputo}
For $f\in\F$, the Caputo integral exists absolutely
at every $t>0$. Its continuous extension at zero
defines an element of $\F$ satisfying
\begin{equation}
\label{eq:frechet-caputo-coefficients}
\CD f(z)
=
\sum_{n=0}^{\infty}
\frac{d_{n+1}}{\Gamma(\alpha n+1)}z^{\alpha n}.
\end{equation}
For every $k\ge0$,
\begin{equation}
\label{eq:frechet-shift-identity}
Q_k(\CD f)=Q_{k+1}(f).
\end{equation}
In particular, $\CD:\F\to\F$ is continuous.
\end{proposition}

\begin{proof}
The absolute continuity established above gives
\[
f'(s)=
\sum_{n=1}^{\infty}\alpha n\,b_ns^{\alpha n-1}
\quad\text{almost everywhere}.
\]
The power formula and Cauchy--Schwarz yield
\[
\begin{aligned}
&\frac1{\Gamma(1-\alpha)}
\sum_{n=1}^{\infty}|\alpha n b_n|
\int_0^t(t-s)^{-\alpha}s^{\alpha n-1}\,ds\\
&\qquad=
\sum_{j=0}^{\infty}
\frac{|d_{j+1}|t^{\alpha j}}
{\Gamma(\alpha j+1)}\\
&\qquad\le
Q_1(f)\,E_\alpha(t^{2\alpha})^{1/2}
<\infty.
\end{aligned}
\]
Thus integration term by term is justified and
gives \eqref{eq:frechet-caputo-coefficients}
on the positive axis.

The shifted series converges compactly on $\Omega$
and continuously at zero.
Its normalized coefficients are $(d_{n+1})_{n\ge0}$,
so
\[
\begin{aligned}
Q_k(\CD f)^2
&=
\sum_{j=0}^{\infty}
\frac{|d_{j+k+1}|^2}{\Gamma(\alpha j+1)}\\
&=
Q_{k+1}(f)^2.
\end{aligned}
\]
Consequently,
\[
\mathfrak P_k(\CD f)
\le
\mathfrak P_{k+1}(f),
\]
which proves continuity.
\end{proof}

The Banach realizations embed compactly into this
Fr\'echet space.

\begin{proposition}[Compact dense embedding]
\label{prop:embedding}
Holomorphic extension defines a continuous injective map
\[
\iota_{\rho,p}:\X\longrightarrow\F
\]
with dense proper range. It maps bounded subsets of $\X$
into relatively compact subsets of $\F$ and satisfies
\begin{equation}
\label{eq:embedding-bound}
Q_k(\iota_{\rho,p}f)
\le
\rho^kE_\alpha(\rho^2)^{1/2}\norm{f}_{\X}.
\end{equation}
For $N\ge k$,
\begin{equation}
\label{eq:embedding-tail}
\sup_{\norm{f}_{\X}\le1}
Q_k\bigl(\iota_{\rho,p}(I-P_N)f\bigr)
\le
\rho^k
\left(
\sum_{j=N-k}^{\infty}
\frac{\rho^{2j}}{\Gamma(\alpha j+1)}
\right)^{1/2}.
\end{equation}
Moreover,
\begin{equation}
\label{eq:embedding-intertwining}
\iota_{\rho,p}\CD
=
\CD\iota_{\rho,p},
\end{equation}
where each Caputo operator acts on its corresponding
space.
\end{proposition}

\begin{proof}
For $f=J_{\alpha,\rho}a$,
\[
b_n=
\frac{a_n\rho^n}{\Gamma(\alpha n+1)},
\qquad
d_n=\rho^na_n.
\]
Therefore,
\[
\begin{aligned}
Q_k(\iota_{\rho,p}f)^2
&=
\rho^{2k}
\sum_{j=0}^{\infty}
\frac{\rho^{2j}|a_{j+k}|^2}
{\Gamma(\alpha j+1)}\\
&\le
\rho^{2k}E_\alpha(\rho^2)\norm{a}_\infty^2\\
&\le
\rho^{2k}E_\alpha(\rho^2)\norm{f}_{\X}^2.
\end{aligned}
\]
This proves \eqref{eq:embedding-bound}.
Restricting the sum to $j+k\ge N$ gives
\eqref{eq:embedding-tail}.

For each fixed $k$, the right-hand side of
\eqref{eq:embedding-tail} tends to zero.
Thus the finite-rank maps $\iota_{\rho,p}P_N$
approximate $\iota_{\rho,p}$ uniformly on the unit
ball in every graph norm.
Their images of the unit ball are bounded subsets
of finite-dimensional spaces.
Hence $\iota_{\rho,p}$ maps the unit ball into
a totally bounded subset of $\F$.
Completeness of $\F$ gives relative compactness.

Coefficient uniqueness proves injectivity.
Fractional polynomials lie in the range and are
dense in $\F$, proving density.

For $|\mu|\ge\rho$, the function
\[
v_\mu(z)=E_\alpha(\mu z^\alpha)
\]
satisfies
\[
Q_k(v_\mu)^2
=
|\mu|^{2k}E_\alpha(|\mu|^2)
<\infty.
\]
However, its coefficient sequence in $\X$ is
$((\mu/\rho)^n)_{n\ge0}$, which does not belong
to $\ell^p$. The range is therefore proper.

Finally, the normalized coefficients of $\CD f$
in both realizations are
\[
d_{n+1}=\rho^{n+1}a_{n+1}.
\]
This proves \eqref{eq:embedding-intertwining}.
\end{proof}

The dense embedding transfers the dynamics once the
radius exceeds the Banach threshold.

\begin{corollary}[Transfer to the Fr\'echet realization]
\label{cor:transfer}
For $0<\alpha<1$ and $\lambda\ne0$,
$\lambda\CD$ is mixing and Devaney chaotic on $\F$.
For every finite sequence
\[
1\le r_1<\cdots<r_N,
\]
the tuple
\[
\bigl((\lambda\CD)^{r_1},\ldots,
      (\lambda\CD)^{r_N}\bigr)
\]
is disjoint mixing on $\F$.
\end{corollary}

\begin{proof}
Choose $\rho>|\lambda|^{-1}$ and fix
$1\le p<\infty$.
Then $|\lambda|\rho>1$, so the Banach realization
is mixing, has dense periodic points, and its
distinct positive powers are disjoint mixing.

Let $V_0,\ldots,V_N$ be nonempty open subsets of
$\F$. Since $\iota_{\rho,p}$ is continuous with
dense range, the sets
\[
U_j=\iota_{\rho,p}^{-1}(V_j)
\]
are nonempty and open in $\X$.
Applying the Banach-space mixing properties to
these preimages and using
\eqref{eq:embedding-intertwining}
gives mixing and disjoint mixing on $\F$.

Each nonempty open preimage also contains a periodic
point. Its image is periodic, so periodic points
are dense in $\F$.
Since $\F$ is metrizable and has no isolated points,
transitivity and dense periodic points imply
sensitivity by \cite{Banks}.
Thus $\lambda\CD$ is Devaney chaotic.
\end{proof}

For $0<\alpha<1$, this recovers the corresponding
Fr\'echet-space conclusions of \cite{MPV} through
bounded Banach realizations. The restriction
$|\lambda|\rho>1$ remains essential on each fixed
space $X_{\alpha,\rho}^{p}$.
% ============================================================
\section{Examples and quantitative consequences}
\label{sec:examples}
% ============================================================

The preceding formulas give explicit perturbations,
approximation times and periodic trajectories.
The examples below illustrate these calculations.

\begin{example}[Explicit costs in an $\ell^1$ realization]
\label{ex:l1-costs}
Fix $0<\alpha<1$, $p=1$ and $\rho=2$.
Then $\CD$ is mixing and Devaney chaotic on
$X_{\alpha,2}^{1}$.

Let
\[
a=\sum_{j=0}^{L}d_je_j,
\qquad d_L\ne0.
\]
For $m>L$, extend $d$ by zeros to length $m$.
The periodic approximation satisfies
\begin{equation}
\label{eq:example-periodic-error}
\norm{\Pi_md-a}_{X_{\alpha,2}^{1}}
=
\frac{\norm{a}_{X_{\alpha,2}^{1}}}{2^m-1}.
\end{equation}
Thus, for any prescribed relative tolerance
$\varepsilon>0$, the choice
\begin{equation}
\label{eq:example-period-choice}
m\ge
\max\left\{
2L+1,
\left\lceil
\log_2(1+\varepsilon^{-1})
\right\rceil
\right\}
\end{equation}
gives a vector of minimal period $m$ such that
\[
\norm{\Pi_md-a}_{X_{\alpha,2}^{1}}
\le
\varepsilon\norm{a}_{X_{\alpha,2}^{1}}.
\]

For exact arrival at $b\in X_{\alpha,2}^{1}$,
take $n>L$. Since $\CD^na=0$,
Theorem~\ref{thm:optimal} gives
\begin{equation}
\label{eq:example-exact-cost}
\min\left\{
\norm{h}_{X_{\alpha,2}^{1}}:
\CD^n(a+h)=b
\right\}
=
2^{-n}\norm{b}_{X_{\alpha,2}^{1}}.
\end{equation}
The unique minimizer is $h=S^nb$.

For comparison, the balanced approximation cost is
\begin{equation}
\label{eq:example-balanced-cost}
\Delta_{\CD,n}(a,b)
=
\frac{\norm{b}_{X_{\alpha,2}^{1}}}{1+2^n}.
\end{equation}

For instance, let $a=e_0+e_1$, $b=e_0$ and $n=2$.
The exact-arrival perturbation is
\[
h_{\mathrm{ex}}=\frac14e_2,
\qquad
\CD^2(a+h_{\mathrm{ex}})=e_0,
\qquad
\norm{h_{\mathrm{ex}}}_{X_{\alpha,2}^{1}}=\frac14.
\]
The balanced minimizer is
\[
h_{\mathrm{bal}}=\frac15e_2,
\]
for which
\[
\norm{h_{\mathrm{bal}}}_{X_{\alpha,2}^{1}}
=
\norm{\CD^2(a+h_{\mathrm{bal}})-e_0}_{X_{\alpha,2}^{1}}
=
\frac15.
\]
\end{example}

Geometric coefficient tails give explicit bounds
on the required iteration time.

\begin{example}[An explicit approximation time]
\label{ex:approximation-time}
Assume $|c|>1$ and choose $s\in\C$ with
$0<|s|<1$. Set
\[
f(t)
=
E_\alpha(\rho s t^\alpha)
=
\sum_{j=0}^{\infty}s^je_j(t).
\]
Then
\begin{equation}
\label{eq:geometric-tail}
\norm{f}_{\X}
=
(1-|s|^p)^{-1/p},
\qquad
\norm{(I-P_N)f}_{\X}
=
|s|^N\norm{f}_{\X}.
\end{equation}

Let $g\in\X\setminus\{0\}$ and $\varepsilon>0$.
Define
\[
M_\varepsilon
=
\max\left\{
\frac{\log(2\norm{f}_{\X}/\varepsilon)}
{\log(1/|s|)},
\frac{\log(2\norm{g}_{\X}/\varepsilon)}
{\log|c|}
\right\},
\]
and choose
\begin{equation}
\label{eq:explicit-arrival-time}
N=
\max\left\{
1,\,
1+\lfloor M_\varepsilon\rfloor
\right\}.
\end{equation}
Then
\[
|s|^N\norm{f}_{\X}<\frac{\varepsilon}{2},
\qquad
|c|^{-N}\norm{g}_{\X}<\frac{\varepsilon}{2}.
\]
Consequently,
\[
z_N=P_Nf+S_\lambda^Ng
\]
satisfies
\begin{equation}
\label{eq:explicit-arrival-estimate}
\norm{z_N-f}_{\X}<\varepsilon,
\qquad
A^Nz_N=g.
\end{equation}

The balanced cost is also explicit. Since
\[
Af=cs\,f,
\]
Theorem~\ref{thm:local-cost} gives
\begin{equation}
\label{eq:geometric-balanced-cost}
\Delta_{A,N}(f,g)
=
\frac{\norm{g-(cs)^Nf}_{\X}}{1+|c|^N}.
\end{equation}
Thus the two approximation problems are controlled
by the same geometric coefficient decay.
\end{example}

The periodic formulas also yield familiar functions
with nontrivial fractional trajectories.

\begin{example}[Exponential cycles]
\label{ex:exponential-cycles}
Let $m\ge2$, $\alpha=1/m$, $\lambda=1$ and $\rho>1$.
The exponential function has the coefficient expansion
\[
e^t
=
\sum_{k=0}^{\infty}\frac{t^k}{k!}
=
\sum_{k=0}^{\infty}\rho^{-mk}e_{mk}(t).
\]
Hence
\begin{equation}
\label{eq:exponential-norm}
e^t\in X_{1/m,\rho}^{p},
\qquad
\norm{e^t}_{X_{1/m,\rho}^{p}}
=
(1-\rho^{-mp})^{-1/p}.
\end{equation}
Proposition~\ref{prop:periodforms} gives
\[
\mathcal C_{1/m}^{\,m}e^t=e^t,
\]
with minimal period $m$.

For $m=2$, write
\[
v(t)=\mathcal C_{1/2}e^t.
\]
Directly from the Caputo integral,
\[
\begin{aligned}
v(t)
&=
\frac1{\sqrt{\pi}}
\int_0^t(t-s)^{-1/2}e^s\,ds\\
&=
\frac{2e^t}{\sqrt{\pi}}
\int_0^{\sqrt t}e^{-u^2}\,du\\
&=
e^t\operatorname{erf}(\sqrt t),
\end{aligned}
\]
where
\[
\operatorname{erf}(x)
=
\frac2{\sqrt{\pi}}\int_0^x e^{-u^2}\,du.
\]
Its fractional power expansion is
\[
v(t)
=
\sum_{k=0}^{\infty}
\frac{t^{k+1/2}}{\Gamma(k+3/2)}
=
\sum_{k=0}^{\infty}
\rho^{-(2k+1)}e_{2k+1}(t).
\]
Therefore,
\[
\mathcal C_{1/2}v=e^t,
\qquad
\norm{v}_{X_{1/2,\rho}^{p}}
=
\rho^{-1}(1-\rho^{-2p})^{-1/p}.
\]
Since $v(0)=0$ and $e^0=1$, the two points are distinct:
\[
e^t
\ \xmapsto{\ \mathcal C_{1/2}\ }\ 
e^t\operatorname{erf}(\sqrt t)
\ \xmapsto{\ \mathcal C_{1/2}\ }\ 
e^t.
\]
\end{example}

Finally, the same differential expression can exhibit
different dynamics on nested Banach spaces.

\begin{example}[A dense stable subspace of a chaotic realization]
\label{ex:ambient-topology}
Fix $0<\alpha<1$ and $1\le p<\infty$.
Let $T_1$ and $T_2$ denote the unscaled Caputo
realizations on $X_{\alpha,1}^{p}$ and
$X_{\alpha,2}^{p}$, respectively.

The inclusion
\[
\iota:X_{\alpha,1}^{p}
\hookrightarrow X_{\alpha,2}^{p}
\]
is compact, has dense proper range, and satisfies
\[
\norm{\iota}=1,
\qquad
T_2\iota=\iota T_1.
\]
For every $f\in X_{\alpha,1}^{p}$,
\[
\norm{T_1^nf}_{X_{\alpha,1}^{p}}
\longrightarrow0.
\]
Consequently,
\[
\norm{T_2^n\iota f}_{X_{\alpha,2}^{p}}
=
\norm{\iota T_1^nf}_{X_{\alpha,2}^{p}}
\le
\norm{T_1^nf}_{X_{\alpha,1}^{p}}
\longrightarrow0.
\]

Although $T_2$ is mixing and Devaney chaotic,
\begin{equation}
\label{eq:stable-subspace-no-hypercyclicity}
\iota(X_{\alpha,1}^{p})\cap HC(T_2)
=
\varnothing.
\end{equation}
Likewise,
\begin{equation}
\label{eq:stable-subspace-periodic}
\iota(X_{\alpha,1}^{p})\cap\Per(T_2)
=
\{0\}.
\end{equation}
Indeed, convergence to zero excludes a dense orbit,
and a periodic orbit converging to zero is identically
zero.

Thus a compactly embedded dense invariant subspace
can consist entirely of stable vectors, while the
ambient realization has dense hypercyclic and
periodic sets.
\end{example}

% ============================================================
\section{Conclusion}\label{sec:conclusion}
% ============================================================

The coefficient realization separates the analytic definition
of Caputo differentiation from its shift dynamics.
The exact local-defect formula quantifies the perturbations
required for approximate and exact arrival, while periodic
blocks provide explicit approximation rates and prescribed
large minimal periods.

On each fixed space $X_{\alpha,\rho}^{p}$, the condition
$|\lambda|\rho>1$ characterizes hypercyclicity, mixing
and Devaney chaos, and determines the same threshold
for disjoint mixing of distinct positive powers.
Compact dense embeddings connect these bounded realizations
to the Mittag--Leffler--Caputo Fr\'echet space.

The local-defect characterization remains a norm formulation
of Birkhoff transitivity. Its role here is quantitative:
it converts coefficient estimates into explicit approximation
costs. Extending these calculations to nonconstant coefficient
weights and perturbed fractional operators is a natural
direction for further study.

% ============================================================
% Bibliography
% ============================================================


\begingroup
\small

\begin{thebibliography}{99}

\bibitem{AN2026}
N. Adhikary, A. Nagar,
On certain forms of transitivities for linear operators,
arXiv:2603.25073v1, 2026.
\url{https://doi.org/10.48550/arXiv.2603.25073}.

\bibitem{Banks}
J. Banks, J. Brooks, G. Cairns, G. Davis, P. Stacey,
On Devaney's definition of chaos,
Amer. Math. Monthly 99 (1992) 332--334.
\url{https://doi.org/10.1080/00029890.1992.11995856}.

\bibitem{BM}
F. Bayart, \'E. Matheron,
Dynamics of Linear Operators,
Cambridge Tracts in Mathematics, vol.~179,
Cambridge University Press, Cambridge, 2009.
\url{https://doi.org/10.1017/CBO9780511581113}.

\bibitem{BernardesPeris2024}
N.C. Bernardes Jr., A. Peris,
On shadowing and chain recurrence in linear dynamics,
Adv. Math. 441 (2024) 109539.
\url{https://doi.org/10.1016/j.aim.2024.109539}.

\bibitem{BMPS}
J. B\`es, \"O. Martin, A. Peris, S. Shkarin,
Disjoint mixing operators,
J. Funct. Anal. 263 (2012) 1283--1322.
\url{https://doi.org/10.1016/j.jfa.2012.05.018}.

\bibitem{BP}
J. B\`es, A. Peris,
Hereditarily hypercyclic operators,
J. Funct. Anal. 167 (1999) 94--112.
\url{https://doi.org/10.1006/jfan.1999.3437}.

\bibitem{BPeris}
J. B\`es, A. Peris,
Disjointness in hypercyclicity,
J. Math. Anal. Appl. 336 (2007) 297--315.
\url{https://doi.org/10.1016/j.jmaa.2007.02.043}.

\bibitem{CEMP2026}
C. Cobollo, R. Ernst, Q. Menet, A. Peris,
Hypercyclicity and Lipschitz-free operators,
arXiv:2609.17874v1, 2026.
\url{https://doi.org/10.48550/arXiv.2609.17874}.

\bibitem{DeLaRosaRead}
M. de la Rosa, C. Read,
A hypercyclic operator whose direct sum
$T\oplus T$ is not hypercyclic,
J. Operator Theory 61 (2009) 369--380.
\url{https://jot.theta.ro/jot/archive/2009-061-002/2009-061-002-010.html}.

\bibitem{GS}
R.M. Gethner, J.H. Shapiro,
Universal vectors for operators on spaces of
holomorphic functions,
Proc. Amer. Math. Soc. 100 (1987) 281--288.
\url{https://doi.org/10.1090/S0002-9939-1987-0884467-4}.

\bibitem{GodefroyShapiro}
G. Godefroy, J.H. Shapiro,
Operators with dense, invariant, cyclic vector manifolds,
J. Funct. Anal. 98 (1991) 229--269.
\url{https://doi.org/10.1016/0022-1236(91)90078-J}.

\bibitem{GG}
D. Gomes, K.-G. Grosse-Erdmann,
Kitai's Criterion for composition operators,
J. Math. Anal. Appl. 547 (2025) 129347.
\url{https://doi.org/10.1016/j.jmaa.2025.129347}.

\bibitem{Gorenflo}
R. Gorenflo, A.A. Kilbas, F. Mainardi, S.V. Rogosin,
Mittag-Leffler Functions, Related Topics and Applications,
second ed., Springer Monographs in Mathematics,
Springer, Berlin--Heidelberg, 2020.
\url{https://doi.org/10.1007/978-3-662-61550-8}.

\bibitem{Grivaux}
S. Grivaux,
Hypercyclic operators, mixing operators,
and the bounded steps problem,
J. Operator Theory 54 (2005) 147--168.
\url{https://jot.theta.ro/jot/archive/2005-054-001/2005-054-001-010.html}.

\bibitem{GEWeightedShifts}
K.-G. Grosse-Erdmann,
Hypercyclic and chaotic weighted shifts,
Studia Math. 139 (2000) 47--68.
\url{https://doi.org/10.4064/sm-139-1-47-68}.

\bibitem{GEP}
K.-G. Grosse-Erdmann, A. Peris Manguillot,
Linear Chaos,
Universitext, Springer, London, 2011.
\url{https://doi.org/10.1007/978-1-4471-2170-1}.

\bibitem{Kitai}
C. Kitai,
Invariant Closed Sets for Linear Operators,
Ph.D. thesis, University of Toronto, Toronto, 1982.

\bibitem{LMP}
C. Lizama, M. Murillo-Arcila, A. Peris,
Nonlocal operators are chaotic,
Chaos 30 (2020) 103126.
\url{https://doi.org/10.1063/5.0018408}.

\bibitem{MalikThome}
S.B. Malik, N. Thome,
On a revisited Moore--Penrose inverse
of a linear operator on Hilbert spaces,
Filomat 31 (2017) 1927--1931.
\url{https://doi.org/10.2298/FIL1707927M}.

\bibitem{MPV}
M. Murillo-Arcila, A. Peris, \'A. Vargas-Moreno,
Dynamics of the Caputo fractional derivative,
Fract. Calc. Appl. Anal. 28 (2025) 1717--1731.
\url{https://doi.org/10.1007/s13540-025-00430-4}.

\bibitem{Podlubny}
I. Podlubny,
Fractional Differential Equations,
Mathematics in Science and Engineering, vol.~198,
Academic Press, San Diego, 1999.

\bibitem{Rolewicz}
S. Rolewicz,
On orbits of elements,
Studia Math. 32 (1969) 17--22.
\url{https://doi.org/10.4064/sm-32-1-17-22}.

\bibitem{RD}
J.A. Rosenfeld, W.E. Dixon,
Approximating the Caputo fractional derivative through
the Mittag-Leffler reproducing kernel Hilbert space
and the kernelized Adams--Bashforth--Moulton method,
SIAM J. Numer. Anal. 55 (2017) 1201--1217.
\url{https://doi.org/10.1137/16M1056894}.

\bibitem{RRD}
J.A. Rosenfeld, B. Russo, W.E. Dixon,
The Mittag Leffler reproducing kernel Hilbert spaces
of entire and analytic functions,
J. Math. Anal. Appl. 463 (2018) 576--592.
\url{https://doi.org/10.1016/j.jmaa.2018.03.036}.

\bibitem{Salas}
H.N. Salas,
Hypercyclic weighted shifts,
Trans. Amer. Math. Soc. 347 (1995) 993--1004.
\url{https://doi.org/10.1090/S0002-9947-1995-1249890-6}.

\bibitem{Samko}
S.G. Samko, A.A. Kilbas, O.I. Marichev,
Fractional Integrals and Derivatives:
Theory and Applications,
Gordon and Breach Science Publishers, Yverdon, 1993.

\bibitem{ShkarinKitai}
S. Shkarin,
The Kitai criterion and backward shifts,
Proc. Amer. Math. Soc. 136 (2008) 1659--1670.

\end{thebibliography}
\endgroup
\end{document}